\documentclass[11pt]{amsart}

\theoremstyle{plain}
\newtheorem{thm}{Theorem}[section]
\newtheorem{theorem}[thm]{Theorem}

\newtheorem{lemma}[thm]{Lemma}
\newtheorem{corollary}[thm]{Corollary}
\newtheorem{proposition}[thm]{Proposition}
\theoremstyle{definition}

\newtheorem{remark}[thm]{Remark}

\newtheorem{definition}[thm]{Definition}

\newtheorem{example}[thm]{Example}

\numberwithin{equation}{section}

\newcommand{\sC}{{\mathcal C}}

\newcommand{\sE}{{\mathcal E}}
\newcommand{\sF}{{\mathcal F}}

\newcommand{\sH}{{\mathcal H}}

\newcommand{\sJ}{{\mathcal J}}

\newcommand{\sN}{{\mathcal N}}
\newcommand{\sO}{{\mathcal O}}

\newcommand{\sR}{{\mathcal R}}

\newcommand{\sU}{{\mathcal U}}
\newcommand{\sV}{{\mathcal V}}
\newcommand{\sW}{{\mathcal W}}

\newcommand{\sY}{{\mathcal Y}}
\newcommand{\sZ}{{\mathcal Z}}

\newcommand{\C}{{\mathbb C}}

\newcommand{\J}{{\mathbb J}}

\newcommand{\BP}{{\mathbb P}}

\newcommand{\X}{{\mathbb X}}

\title[Geometry of webs of algebraic curves]{Geometry of webs of  algebraic curves}
\author[Jun-Muk Hwang]{Jun-Muk Hwang} 
\address{Korea Institute for Advanced Study, Hoegiro 87, Seoul 02455, Korea}
\email{jmhwang@kias.re.kr}
\thanks{The author is supported
by National Researcher Program 2010-0020413 of NRF}

\begin{document}

\begin{abstract} A family of algebraic curves covering a projective variety $X$ is called a web of curves on $X$ if it has only finitely many members through a general point of $X$. A web of curves on $X$ induces a web-structure, in the sense of local differential geometry, in a neighborhood of a general point of $X$.  We study how the local differential geometry of the web-structure affects the global algebraic geometry of $X$. Under two geometric assumptions on the web-structure, the pairwise non-integrability condition and the bracket-generating condition,  we prove that the local differential geometry determines the global algebraic geometry of $X$, up to generically finite algebraic correspondences.  The two geometric assumptions are satisfied, for example, when $X \subset \BP^N$ is a Fano submanifold of  Picard number 1,  and the family of lines covering $X$ becomes a web.   In this special case, we have a stronger result that the local differential geometry of the web-structure determines $X$ up to biregular equivalences. As an application, we show that if $X, X' \subset \BP^N, \dim X' \geq 3,$ are two such Fano manifolds of Picard number 1,  then any surjective morphism $f: X \to X'$ is an isomorphism.
\end{abstract}

\maketitle



\section{Introduction}
Consider families of algebraic curves covering a projective variety $X$ in such a way that there are only finitely many members of the family through a general point of $X.$  We will call such a family  a `web of curves' (Definition \ref{d.web}) on $X$.
 In a Euclidean neighborhood of a general point of $X$, a web of curves on $X$ induces a `web-structure' (Definition \ref{d.regular}), i.e. a finite collection of (1-dimensional) holomorphic foliations,   a classical object in  differential geometry.
  Although the study of web-structures has a long history in differential geometry (see \cite{PP} and the references therein.), most of the existing theory is about web-structures of codimension 1.
  In this article, we will investigate {\em how the  local differential geometry of  the web-structure induced  by a web of curves affects the   global algebraic geometry} of the projective variety $X$.
Our work suggests that the theory of 1-dimensional web-structures on  manifolds of dimension $\geq 3$ is a worthy subject of study.

\medskip
 The original motivation of this work was to prove the following.

\begin{theorem}\label{t.ultimate}
Let $X, X' \subset \BP^N$  be two projective submanifolds of Picard number 1 covered by lines of $\BP^N$.
Let $\varphi: U \to U'$ be a biholomorphic map between two connected Euclidean open subsets $U \subset X$ and $U' \subset X'$
such that  $\varphi$ (resp. $\varphi^{-1}$) sends germs of lines in $U$ (resp. $U'$) to germs of lines in $U'$ (resp. $U$).
Then there exists a biholomorphic map (i.e. a biregular morphism)  $\Phi: X \to X'$ such that $\varphi = \Phi|_{U}$. \end{theorem}

This was proved in \cite{HM01}  under the assumption that the family of lines passing through a general point of $X$ and $X'$ has positive dimension.
 The remaining part of Theorem \ref{t.ultimate}, for which the method of \cite{HM01} fails, is exactly when the families of lines on $X$ and $X'$ form webs of curves.  This remaining part (and  a more  general version) has been raised as an open question  in p. 566 of \cite{HM01}   and appeared as Question 5 in \cite{Hw} in the list of major open problems in the study of minimal rational curves. Our Theorem \ref{t.ultimate} settles this remaining part. More explicitly, we can state the new component of Theorem \ref{t.ultimate} as follows.

\begin{theorem}\label{t.ultim}
Let $X, X' \subset \BP^N,  \dim X = \dim X',$ be two projective manifolds of Picard number 1 through a general point of which there are only finite, but nonzero,   number of lines. Let $\sW$ (resp. $\sW'$) be a web of curves on $X$ (resp. $X'$) whose members are lines in $\BP^N$.
Let $\varphi: U \to U'$ be a biholomorphic map between two connected Euclidean open subsets $U \subset X$ and $U' \subset X'$
such that  $\varphi$ (resp. $\varphi^{-1}$) sends germs of lines belonging to $\sW$ (resp. $\sW'$) to germs of lines belonging to $\sW'$ (resp. $\sW$).
Then there exists a biholomorphic map (i.e. a biregular morphism)  $\Phi: X \to X'$ such that $\varphi = \Phi|_{U}$. \end{theorem}

The condition that $\varphi$ (resp. $\varphi^{-1}$) sends germs of lines in $U$ (resp. $U'$) to germs of lines in $U'$ (resp. $U$) means that $\varphi$ is an equivalence of the web-structures in the sense of local differential geometry. Thus Theorem \ref{t.ultim} precisely says that, under the given assumptions,  the local equivalence of  the web-structures implies the biregular equivalence of the projective varieties. If we choose $\sW$ and $\sW'$ as the webs of all lines covering the projective manifolds, then Theorem \ref{t.ultim} gives the remaining part of Theorem \ref{t.ultimate}.

It is crucial that the open subsets $U$ and $U'$ in Theorem \ref{t.ultim} are in Euclidean topology.  As a matter of fact, if we replace
Euclidean open subsets by Zariski open subsets in Theorem \ref{t.ultim}, the proof becomes straight-forward (see Proposition \ref{p.easy}).
Thus the key issue in the proof of Theorem \ref{t.ultim} is to extend a holomorphic map defined on a Euclidean open subset to a Zariski open subset.
We will achieve this in two steps:

\smallskip
\textsf{Step 1}.  Extension from a Euclidean open subset to an \'etale open subset.

\textsf{Step 2}.  Extension from an \'etale open subset to  a Zariski open subset.

\smallskip
It turns out that  our argument for \textsf{Step 1} works in a much more general setting than Theorem \ref{t.ultim} and proves the following.

\begin{theorem}\label{t.main}
Let $X$ (resp. $X'$) be  a projective variety with a web $\sW$ (resp. $\sW'$)  of curves. Assume that both $\sW$ and $\sW'$ are
\begin{itemize} \item[(P)] pairwise non-integrable (Definition \ref{d.P}) and \item[(B)] bracket-generating (Definition \ref{d.B}). \end{itemize}
Let $\varphi:U \to U'$ be a biholomorphic map between two connected Euclidean open subsets $U \subset X$ and $U' \subset X'$
such that  $\varphi$ (resp. $\varphi^{-1}$) sends germs of members of $\sW$ in $U$ (resp. $\sW'$ in $U'$) to germs of  members of $\sW'$ in $U'$ (resp. $\sW$  in $U$).
  Then $\varphi$ can be extended to a generically finite algebraic correspondence between $X$ and $X'$, i.e.,  there exists a projective subvariety $\Gamma \subset X \times X'$ which contains the graph of $\varphi$ and is generically finite over both $X$ and $X'$.  \end{theorem}

This says that  we can extend the complex analytic (or differential geometric) equivalence of the web-structures on Euclidean open subsets to an algebraic equivalence on \'etale open subsets, provided  the webs satisfy  two conditions (P) and (B).  Both  conditions are formulated as the failure of the  involutiveness of certain distributions associated with the web-structures induced by the webs of curves. So these conditions are local differential geometric properties of the webs.  But both of them can be  interpreted also as algebro-geometric conditions on the webs of curves (Corollary \ref{c.nonintegrable} and Proposition \ref{p.sat}, respectively).

The following examples show that both conditions (P) and (B) are necessary for Theorem \ref{t.main}. Fix  two domains $O, O' \subset \C$
such that the restriction $f: O \to O'$ of the exponential map $e^z: \C \to \C$ is a biholomorphism.

\begin{example}\label{e.I}
Set $X = X' = \BP^1 \times \BP^1$ and let  $\sW$ be the web consisting of the two irreducible  families of curves
given by each factor of $\BP^1$. This web satisfies (B), but not (P).  Set $U = O \times O$ and $U' = O' \times O'$. Then  the product $(f, f): U \to U'$ is a biholomorphic map
preserving the web-structures, but cannot be extended to a generically finite correspondence.
\end{example}

\begin{example}\label{e.S}
Use the terminology of Theorem \ref{t.main}.
Consider projective varieties  $Y:= \BP^1 \times X$ and $Y':=\BP^1 \times X'$. They are equipped with webs $\sV$ and $\sV'$ induced by $\sW$ and $\sW'$.  Then $\sV$ and $\sV'$ satisfy (P), but not (B). Put $V = O \times U$ and $V'= O' \times U'$. The biholomorphic map $(f, \varphi): V \to V'$  preserves the web-structures, but cannot be extended to a generically finite correspondence between $Y$ and $Y'$.
\end{example}

The condition (B), in a different form,  had  appeared also  in \cite{HM01}  and was used  crucially in the extension argument there. Its role in the current work is very similar to that in \cite{HM01}, based on the construction (see Proposition \ref{p.tower}) of a tower of auxiliary varieties by attaching members of the family of curves in an inductive way.
 A novel part of our argument in \textsf{Step 1}  is to use the condition (P) to overcome the difficulty in applying the method of \cite{HM01} in the current setting. Roughly speaking, the condition (P) provides  the parameter space $\sW$ with a family of curves (see Corollary \ref{c.nonintegrable}) whose members through a general point of $\sW$ form a positive-dimensional family.  This situation is very similar to the main setting of \cite{HM01}, except that these curves are not necessarily determined by their tangent directions, unlike the minimal rational curves considered in \cite{HM01}.  But this technical difference can be handled by using higher jets of curves (see Proposition \ref{p.Psi}) in place of their tangent directions and  we can carry out the extension procedure
in a way analogous to that of \cite{HM01}.

An important class of webs satisfying both (P) and (B)  is  \'etale webs of smooth rational curves (Definition \ref{d.etale}) on Fano manifolds of Picard number 1.  In particular, Theorem \ref{t.main} implies the following general version of Theorem \ref{t.ultim}.

\begin{theorem}\label{t.rational}
Fix two positive integers $\ell, \ell' >0$.
Let $X \subset \BP^N$ (resp. $X' \subset \BP^N$) be a projective manifold of Picard number 1 through a general point of which there are only finite, but nonzero,   number of smooth rational curves of degree $\ell$ (resp. $\ell'$).
Let $\sW$ (resp. $\sW'$) be a  web of  curves on $X$ (resp. $X'$) general members of which are smooth rational curves of degree $\ell$  (resp. $\ell'$).
Let $\varphi: U \to U'$ be a biholomorphic map between two connected Euclidean open subsets $U \subset X$ and $U' \subset X'$
such that  $\varphi$ (resp. $\varphi^{-1}$) sends germs of rational curves belonging to $\sW$ (resp. $\sW'$) to germs of rational curves belonging to $\sW'$ (resp. $\sW$).
 Then $\varphi$ can be extended to a generically finite algebraic correspondence between $X$ and $X'$, i.e.,  there exists an irreducible projective subvariety $\Gamma \subset X \times X'$ which contains the graph of $\varphi$ and is generically finite over both $X$ and $X'$.  \end{theorem}

There are many Fano manifolds of Picard number 1 having such webs. In fact,
   all Fano threefolds of Picard number 1, excepting   the 3-dimensional projective space and the 3-dimensional hyperquadric,  have \'etale webs of smooth rational curves (see Chapter 4 of \cite{IP}).

\medskip
 While \textsf{Step 1} of the proof of Theorem \ref{t.ultim} works in the general setting of Theorem \ref{t.main}, the argument in \textsf{Step 2} for  the extension from an \'etale open subset to a Zariski open subset is more subtle and  does not work even in the setting of Theorem \ref{t.rational}. To allow the argument in \textsf{Step 2}, the web should not be `pleated' (see Definition \ref{d.pleat}), which is
a global algebro-geometric condition. We verify this condition for Theorem \ref{t.ultim} by exploiting a deformation-theoretic property of lines,  which does not hold for rational curves of higher degree. In fact, the following example shows that we cannot expect  $\Gamma$ in Theorem \ref{t.rational} with $\ell, \ell' >1$ to be the graph of a biregular morphism or even a rational map.

\begin{example}\label{e.cubic}
Let $X_0 \subset \BP^4$ be a smooth cubic threefold. There are exactly six lines through a general point of $X_0$.
Choose two general quadric hypersurfaces $Q, Q' \subset \BP^4$ and let $f: X \to X_0$ (resp. $f': X' \to X_0$) be the double cover of $X_0$ branched along $Q\cap X_0$ (resp. $Q' \cap X_0).$  Then $X, X'$ are Fano threefolds of Picard number 1 and the inverse images of
lines on $X_0$ give rise to an \'etale web $\sW$ (resp. $\sW'$) of smooth rational curves on $X$ (resp. $X'$).  We can choose connected
Euclidean open subsets $U_0 \subset X_0, U \subset X$ and $U' \subset X'$ such that $U_0= f(U) = f'(U')$ and the restrictions
$$U \stackrel{f|_{U}}{\longrightarrow} U_0 \stackrel{f'|_{U'}}{\longleftarrow} U'$$ are biholomorphic.
Then the composition $$\varphi := (f'|_{U'})^{-1} \circ f|_{U}: U \to U'$$ is a biholomorphic map sending germs of
members of  $\sW$ to those of $\sW'$. For general choices of $Q$ and $Q'$, the two varieties $X$ and $X'$ cannot be biregular. The two morphisms $f$ and $f'$ give rise to a generically finite correspondence, predicted by Theorem \ref{t.rational}, between $X$ and $X'$. \end{example}

The argument of \textsf{Step 2} has the following application.

\begin{theorem}\label{t.application}
Fix a positive integer $\ell' >0$.
Let $X \subset \BP^N$ (resp. $X' \subset \BP^N$) be a projective manifold of Picard number 1 through a general point of which there are only finite, but nonzero,   number of lines (resp. smooth rational curves of degree  $\ell'$).
 Assume that $\dim X = \dim X' \geq 3.$  Then any surjective morphism $f: X \to X'$ is an isomorphism.
\end{theorem}

Some special cases of Theorem \ref{t.application} have been  known.  \cite{Sc} proved it  when $\dim X = \dim X'= 3$  under the additional assumption of the smoothness of  the Hilbert scheme of lines.  \cite{HM03} proved it  when $X=X'$, i.e., when $f$ is a self-map.  These special cases have been handled by arguments quite different from ours. To our knowledge, Theorem \ref{t.application} is new even when $X, X'$ are Fano complete intersections of index 2 of dimension $\geq 3$, which always satisfy the conditions of Theorem \ref{t.application}.

\medskip
This paper is organized as follows.
In Section \ref{s.curves}, we present the basic definitions and some general results on covering families of curves.
We introduce the condition (P) in Section \ref{s.P} and the  condition (B) in Section \ref{s.B}.   The proof of Theorem \ref{t.main} is given  in Section \ref{s.proof}. Section \ref{s.etale} is the verification that the webs in Theorem \ref{t.rational} satisfy the two conditions.
In Section \ref{s.pleat}, we introduce the concept of pleated webs, which is the key in the argument of \textsf{Step 2}. Using this concept, we will prove Theorem \ref{t.ultim} and Theorem \ref{t.application} in Section \ref{s.line}.

\medskip
{\bf Convention}

 \begin{itemize} \item[1.] We work over the complex numbers. Open sets and neighborhoods refer to Euclidean topology, unless otherwise specified. \item[2.] An analytic (resp. algebraic) variety is an irreducible reduced complex space (resp. algebraic scheme). A Zariski open subset of an analytic variety means the complement of a closed analytic subset.
 \item[3.] Let
$f: X \to Y$ be a  holomorphic map between  varieties. For a nonsingular point $x \in X$ with $y= f(x)$ a nonsingular point of $Y$,
we say that $f$ is unramified (resp. submersive) at  $x$ if the derivative ${\rm d}f_x: T_x(X) \to T_{y}(Y)$ is injective (resp. surjective).   When  $f: X \dasharrow Y$ is a meromorphic map, we say that $f$ is generically submersive if its germ at a general point of $X$  is submersive, and $f$ is generically biholomorphic if its germ at a general point of $X$ is biholomorphic.
\end{itemize}

\section{Covering families of curves}\label{s.curves}

\begin{definition}\label{d.family}
Let $M$ be a projective variety. \begin{itemize} \item[(i)]
A projective subvariety $\sF \subset {\rm Chow}^1(M)$ of the Chow variety of 1-cycles on $M$ is called an {\em irreducible covering family of curves} on $M$  if the following conditions hold for the universal family morphisms $\rho_{\sF}: {\rm Univ}_{\sF} \to {\sF}$  and $\mu_{\sF}: {\rm Univ}_{\sF} \to M$ (see I.3 of \cite{Ko} for the definition of  ${\rm Univ}_{\sF}$):
\begin{itemize}
\item[(1)] a general fiber of $\rho_{\sF}$ is irreducible and reduced; and
 \item[(2)] $\mu_{\sF}$ is surjective.
  \end{itemize}
\item[(ii)]  A {\em covering family of curves} on $M$ means a finite union of irreducible covering families of curves.
\item[(iii)] We will denote by ${\rm Univ}^{\rm sm}_{\sF}$ the dense Zariski  open subset in
 ${\rm Univ}_{\sF}$ consisting of nonsingular points of ${\rm Univ}_{\sF}$ where the morphism $\rho_{\sF}$ is smooth.
 \end{itemize}
\end{definition}


\begin{definition}\label{d.pushpull}
 Let $\sF \subset {\rm Chow}^1(M)$ be a covering family of curves on a projective variety $M$.  \begin{itemize} \item[(i)] For a surjective morphism $g:M \to M^{'}$ to a projective variety $M^{'}$ that  does not contract general members of any irreducible component of $\sF$,  the images under $g$ of the general members of irreducible components of $\sF$  determine a covering family of curves on $M^{'}$ which will be denoted by $g_* \sF \subset {\rm Chow}^1(M^{'})$. We have a natural dominant rational map
 ${\rm univ}_{g_*\sF}: {\rm Univ}_{\sF} \dasharrow {\rm Univ}_{g_*\sF}$.
 \item[(ii)] For a generically finite morphism $f: M^{'} \to M$ from a projective variety $M^{'}$, the inverse images under $f$ of the general members of irreducible components of $\sF$ determine a covering family of curves on $M^{'}$ which will be denoted by $f^*\sF \subset {\rm Chow}^1(M^{'})$.
     We have a natural generically finite rational map ${\rm univ}_{f^*\sF} : {\rm Univ}_{f^*\sF} \dasharrow {\rm Univ}_{\sF}$. We will denote by ${\rm Bir}(f^*\sF)$ (resp. ${\rm Mult}(f^*\sF)$) the union of irreducible components of $f^* \sF$ general members of which are sent to members of $\sF$ birationally (resp. not birationally) by $f$, so that $f^* \sF = {\rm Bir}(f^*\sF) \cup {\rm Mult}(f^*\sF)$.
     \end{itemize}
 Note that in our definitions of $g_*\sF$ and $f^* \sF$, the images and the inverse images of members of $\sF$ are taken in the set-theoretical sense, not in cycle-theoretic sense.
\end{definition}

Now we recall some facts on finite-order jet spaces of curves.
The fiber bundle $\J^k (M)$ in the next definition is a Zariski open subset in the Semple $k$-jet bundle of $T(M)$ in the sense of \cite{De}. We refer those who want a more precise presentation to Section 5 and Section 6 of \cite{De}, but we do not need the structure theory developed there.

\begin{definition}\label{d.jet}
Let $M$ be a complex manifold. For a nonnegative integer $k \geq 0$, two germs  of 1-dimensional submanifolds at a point $x \in M$ are {\em $k$-jet equivalent}  if they have contact order at least $k$ at $x$. \begin{itemize} \item[(1)] Denote by $\J^k_x(M)$ the complex manifold consisting of the $k$-jet equivalence classes  of germs of  1-dimensional submanifolds at $x \in M$.
This complex manifold $\J^k_x(M)$ has a natural structure of quasi-projective algebraic variety.
The  union $\J^k (M) = \cup_{x \in M} \J^k_x(M)$ is  a holomorphic fiber bundle on $M$ with a natural projection
 $\pi_M^k: \J^k (M) \to M$. For example, $\J^0 (M) = M$ and $\J^1 (M) = \BP T(M)$.
  \item[(2)]
  A biholomorphic map between complex manifolds $f: M \to M'$  induces a biholomorphic  fiber bundle morphism ${\rm d}^k f: \J^k (M) \to \J^k (M')$
   for each $k \geq 0$ satisfying the commuting diagram $$ \begin{array}{ccc}  \J^k (M) & \stackrel{{\rm d}^k f}{\longrightarrow} & \J^k (M') \\ \pi^k_M \downarrow & & \downarrow \pi^k_{M'} \\ M & \stackrel{f}{\longrightarrow} & M'. \end{array}$$
 \item[(3)] Let $\overline{M}$ be an analytic variety and let $M$ be its smooth locus. Then there exists an analytic variety $\J^k(\overline{M})$ with a meromorphic map $\pi^k_{\overline{M}}: \J^k(\overline{M}) \dasharrow \overline{M}$ such that
       the restriction of $\pi^k_{\overline{M}}$ to the inverse image of $M$ is naturally isomorphic to $\pi^k_M$.
        A generically biholomorphic meromorphic map $f: \overline{M} \dasharrow \overline{M'}$ between analytic varieties   induces a generically biholomorphic meromorphic map ${\rm d}^k f: \J^k (\overline{M}) \dasharrow \J^k (\overline{M'})$
        which agrees with ${\rm d}^k f_x$ of (2) for the biholomorphic germ $f_x$ of $f$ at a general point $x \in M$.
        %

\end{itemize}
  \end{definition}

 \begin{proposition}\label{p.injection}
 Let $\sF$ be a covering family of curves on a projective variety $M$ with the universal family $\rho_{\sF}: {\rm Univ}_{\sF} \to \sF$ and $\mu_{\sF}: {\rm Univ}_{\sF} \to M$.   Then for each $ k \geq 0$, there exists a natural rational map $j^k_{\sF}: {\rm Univ}_{\sF} \dasharrow  \J^k (M)$ with the commuting diagram $$ \begin{array}{ccc} {\rm Univ}_{\sF} & \stackrel{j^k_{\sF}}{\dasharrow} &
 \J^k (M) \\ \mu_{\sF} \downarrow & & \downarrow \pi^k_M \\ M & = & M.
 \end{array} $$
  Furthermore, for sufficiently large $k$, the rational map $j^k_{\sF}$ is generically injective on ${\rm Univ}_{\sF}$. \end{proposition}

\begin{proof}
The map $j^k_{\sF}$ is defined by considering the $k$-jet equivalence classes of the germs of members of $\sF$ at their nonsingular points. The commuting diagram is  immediate from the definition. The generic injectivity of $j^k_{\sF}$ for large $k$ follows from the fact that
for a fixed bounded family $\sF$ of curves on $M$, smooth germs of their members are determined  by their $k$-jets for a sufficiently large $k$.
\end{proof}

\begin{proposition}\label{p.Psi}
Let $\sF $ (resp. $\sF'$) be a covering family of curves on a projective variety $M$ (resp $M'$) with the universal family
$\rho_{\sF}: {\rm Univ}_{\sF} \to \sF$ (resp. $\rho_{\sF'}: {\rm Univ}_{\sF'} \to \sF'$) and $\mu_{\sF}: {\rm Univ}_{\sF} \to M$ (resp.  $\mu_{\sF'}: {\rm Univ}_{\sF'} \to M'$). Let $\phi: V \dasharrow M'$ and $\psi: {\bf U} \dasharrow {\rm Univ}_{\sF'}$ be  generically biholomorphic meromorphic maps from open subsets  $V \subset M$ and ${\bf U} \subset {\rm Univ}_{\sF}$  such that \begin{itemize}
\item[(1)]  ${\bf U}$  intersects
every irreducible component of ${\rm Univ}_{\sF}$;
\item[(2)] $\mu_{\sF}({\bf U}) \subset V$;  \item[(3)] the diagram $$ \begin{array}{ccccc} {\rm Univ}_{\sF} & \supset & {\bf U} & \stackrel{\psi}{\dasharrow} &  {\rm Univ}_{\sF'} \\
\mu_{\sF} \downarrow & &  \downarrow & & \downarrow \mu_{\sF'} \\
M & \supset & V & \stackrel{\phi}{\dasharrow}  & M' \end{array} $$ commutes; and
\item[(4)] $\phi \left( \mu_{\sF}(\rho_{\sF}^{-1}(\rho_{\sF}(y)) \cap {\bf U})\right) \subset \mu_{\sF'}\left(\rho_{\sF'}^{-1} (\rho_{\sF'} (\psi (y)))  \right)$ for  a general $ y \in {\bf U}$. \end{itemize}
Then there exists  a generically biholomorphic meromorphic map
$\Psi: \mu_{\sF}^{-1}(V) \dasharrow {\rm Univ}_{\sF'}$ such that $$ \begin{array}{ccccc} {\rm Univ}_{\sF} & \supset & \mu_{\sF}^{-1}(V) & \stackrel{\Psi}{\dasharrow} & {\rm Univ}_{\sF'} \\
\mu_{\sF} \downarrow & &  \downarrow & & \downarrow \mu_{\sF'} \\
M & \supset & V & \stackrel{\phi}{\dasharrow}  & M' \end{array} $$ commutes and $\Psi|_{{\bf U}} = \psi.$ \end{proposition}

\begin{proof} As in Definition \ref{d.jet} (3),
the meromorphic map $\phi$ induces a meromorphic map
  $${\rm d}^k \phi: \J^k (V) \dasharrow \J^k(M')$$
  such that the induced map on the fiber $${\rm d}^k_x \phi: \J^k_x(V) \longrightarrow \J^k_{\phi(x)}(M')$$ is biregular for a general $x \in V$.
  Let $j_{\sF}^k: {\rm Univ}_{\sF} \dasharrow  \J^k (M)$ (resp. $j_{\sF'}^k: {\rm Univ}_{\sF'} \dasharrow \J^k (M')$) be the rational map defined  in Proposition \ref{p.injection}. Fix $k>>0$ such that both $j^k_{\sF}$ and $j^k_{\sF'}$ are generically injective.
  The conditions (2)-(4) on $\psi$ imply that the following diagram of meromorphic maps commute:
$$   \begin{array}{ccccccc} {\bf U} & \subset & {\rm Univ}_{\sF} & \stackrel{j^k_{\sF}}{\dashrightarrow} & \J^k (M) & \supset & \J^k (V) \\
\psi \downarrow & & & & & & \downarrow {\rm d}^k \phi \\ {\rm Univ}_{\sF'} & = &  {\rm Univ}_{\sF'} & \stackrel{j^k_{\sF'}}{\dashrightarrow} &  \J^k (M') & = & \J^k (M'). \end{array} $$
Since ${\rm d}^k\phi$ is biregular over a general point of $V$, it gives a generically biholomorphic meromorphic map between the proper images $$j^k_{\sF}({\bf U}) \mbox{ and } j^k_{\sF'}(\psi({\bf U})).$$ By the condition (1), it  gives a generically biholomorphic  meromorphic map from the proper image  $j^k_{\sF}({\rm Univ}_{\sF}) \cap (\pi_M^k)^{-1}(V)$ to $  j^k_{\sF'}({\rm Univ}_{\sF'}).$ As $j^k_{\sF}$ and $j^k_{\sF'}$ are generically injective, we have a meromorphic map $\Psi: \mu_{\sF}^{-1}(V) \dasharrow {\rm Univ}_{\sF'}$ with the desired properties.
\end{proof}

\section{Pairwise non-integrable webs}\label{s.P}
\begin{definition}\label{d.web} Let $M$ be a projective variety. \begin{itemize} \item[(i)]
A covering family $\sW$ of curves on $M$ is called a {\em web of curves} if the universal family morphism
$\mu_{\sW}: {\rm Univ}_{\sW} \to M$ is generically finite. Thus a web of curves $\sW$ has pure dimension equal to $\dim M -1$.
\item[(ii)] Given two webs of curves $\sW$ and $\sV$ on $M$, we say that $\sV$ is a {\em subweb} of $\sW$ if $\sV \subset \sW$ as subsets of ${\rm Chow}^1(M)$.  \item[(iii)]
A web $\sW$ of curves is  {\em univalent} if $\mu_{\sW}$ is birational.  \item[(iv)] Any surjective morphism $f: M \to B$ to a projective variety $B$ with $\dim M - \dim B =1$ determines canonically an irreducible univalent web of curves on $M$, to be denoted by ${\rm Fib}(f)$, a general member of which is an irreducible component of a general fiber of $f.$  \end{itemize} \end{definition}

\begin{remark}
Our definition of a web of curves is more general than the one used in \cite{HM03}.  The definition of a web given in \cite{HM03} corresponds to an \'etale web, to be introduced in Section \ref{s.etale}. \'Etale webs are much more restrictive, although our main applications, Theorem \ref{t.ultim} and Theorem \ref{t.rational},  are concerned with them. \end{remark}

The following is immediate by dimension-counting.

\begin{lemma}\label{l.choose}
Let $\sW$ be a web of curves on a projective variety $M$ with the universal family morphisms $\mu_{\sW}: {\rm Univ}_{\sW} \to M$ and    $\rho_{\sW}: {\rm Univ}_{\sW} \to \sW$. For any dense Zariski open subset $\sW_o \subset \sW$, there exists a dense Zariski open subset $M_o \subset M$ such that any member $C$ of $\sW$ satisfying $C \cap M_o \neq \emptyset$ belongs to $\sW_o$. \end{lemma}

The term `web' originates from the notion of a web-structure in local differential geometry, defined as follows.

\begin{definition}\label{d.regular} Let $U$ be a complex manifold.
A {\em web-structure} (of rank 1) on $U$ is a finite collection of line
subbundles $$W_i \subset T(U), \;  1 \leq i \leq d$$
for some integer $d \geq 1$ such that for any  $1 \leq i \neq j \leq d,$ the intersection
$W_i\cap W_j \subset T(U)$ is the zero section.
If we  regard $W_i$ as a 1-dimensional foliation on $U$,
the condition implies that the leaves of $W_i$ and $W_j$ intersect transversally.
\end{definition}

\begin{proposition}\label{p.regular} Let $\sW$ be a web of curves on a projective variety $M$ with the universal family morphisms $\mu_{\sW}: {\rm Univ}_{\sW} \to M$ and
     $\rho_{\sW}: {\rm Univ}_{\sW} \to \sW$.
 Let $d$ be the degree of  $\mu_{\sW}$. Then there exists a Zariski open subset $M_{\rm reg} \subset M$ such that each $x \in M_{\rm reg}$ has an open neighborhood ${\rm Reg}(x) \subset M_{\rm reg}$ satisfying the following conditions. \begin{itemize}
\item[(i)]    $\mu^{-1}_{\sW} ({\rm Reg}(x))$ consists of $d$  disjoint connected open subsets $$R_1, \ldots, R_d \ \subset {\rm Univ}^{\rm sm}_{\sW}$$ each of which is biholomorphic to ${\rm Reg}(x)$ by $\mu_{\sW}$.
\item[(ii)] For each $1 \leq i \leq d$,  each fiber of $\rho_{\sW}|_{R_i}$ is connected.
\item[(iii)] Let  $W_i \subset T({\rm Reg}(x)),  1 \leq i \leq d,$ be the image of tangents to fibers of $\rho_{\sW}|_{R_i}$. Then
 $W_1, \ldots, W_d$ give a web-structure on ${\rm Reg}(x)$. \end{itemize}
 In particular, for any member $C$ of $\sW$, the intersection $C \cap {\rm Reg}(x)$ is smooth. \end{proposition}

\begin{proof}
The existence of a neighborhood ${\rm Reg}(x)$ of a general point $x \in M$ satisfying (i) and (ii) is immediate from Definition \ref{d.web}.
The line subbundles $W_i$ defined  in (iii) are distinct because $\sW \subset {\rm Chow}^1(M)$. Thus  we can achieve (iii) by choosing $M_{\rm reg}$
suitably. \end{proof}

A fundamental notion in the local differential geometry  is the equivalence of  web-structures in the following sense.

\begin{definition}\label{d.preserve} Let $U$ (resp. $U'$) be a complex manifold with a
web-structure $W_i \subset T(U), \;  1 \leq i \leq d$ (resp. $W'_j \subset T(U'), \;  1 \leq j \leq d'$).
A biholomorphic map $\varphi: U \to U'$ {\em sends the web-structure} $W_i, 1 \leq i \leq d$ {\em into the web-structure}
 $W'_j, 1 \leq j \leq d'$  if ${\rm d} \varphi: T(U) \to T(U')$
sends each $W_i$ to some $W'_j$. If furthermore $d = d'$, then we say that $\varphi$ is {\em an equivalence of the web-structures}.
\end{definition}

\begin{definition}\label{d.respect}
Let $\sW$ (resp. $\sW'$) be a web of curves on a projective variety $M$ (resp. $M'$).
Let $\varphi: U \to U'$ be a biholomorphic map between two open sets $U \subset M$ and
$U' \subset M'$.   For each point  $$x \ \in  \ U \cap M_{\rm reg} \cap \varphi^{-1}(M'_{\rm reg}), $$  we can choose
a neighborhood $U^x \subset U \cap {\rm Reg}(x)$  of $x$ (resp. $U^{\varphi(x)} \subset U' \cap {\rm Reg}(\varphi(x))$) with the  web-structure induced by $\sW$ (resp $\sW'$) as in Proposition \ref{p.regular}. By shrinking $U^x$ if necessary, we can assume that $\varphi(U^x) \subset U^{\varphi(x)}$.
 We say that $\varphi$ {\em sends} $\sW$ {\em into} $\sW'$ if the biholomorphic map $\varphi|_{U^x}: U^x \to \varphi(U^x)$ sends the web-structure induced by $\sW$ into the web-structure induced by $\sW'$ for some (hence any by analyticity) $x$ as above. We say that a meromorphic map $\psi: U \dasharrow U'$ sends $\sW$ into $\sW'$ if its biholomorphic germs at general points do so.
  \end{definition}

\begin{definition}\label{d.P}
A  web-structure $W_i \subset T(U), \;  1 \leq i \leq d$, on a complex manifold $U$ is {\em pairwise non-integrable}
if for each $i$, there exists $j \neq i$ such that the distribution $W_i + W_j \subset T(U)$ is not integrable. A web of curves $\sW$ on a projective variety $X$ is {\em pairwise non-integrable}  if the web-structure in ${\rm Reg}(x)$ for each $x \in X_{\rm reg}$ in
Proposition \ref{p.regular} is pairwise non-integrable. \end{definition}

\begin{remark}
The term `pairwise non-integrable' can be confusing, as it may suggest something different from Definition \ref{d.P}, for example, that  $W_i + W_j$ is not integrable for each (or some) pair $(i, j), i \neq j$.
Since we do not have a good alternative term, we will use it by an abuse of language. \end{remark}

\begin{definition}\label{d.ram}
     Let $\sW$ be a web of curves on a projective variety $X$ with the universal family morphisms $\mu_{\sW}: {\rm Univ}_{\sW} \to X$ and
     $\rho_{\sW}: {\rm Univ}_{\sW} \to \sW$.  Fix an  irreducible component $\sV$ of $\sW$ and write $f= \mu_{\sV}, g= \rho_{\sV}$ and $A= {\rm Univ}_{\sV}$. The morphism $f: A \to X$ is generically finite  and $g: A \to \sV$ has connected 1-dimensional fibers.    The inverse image $f^{*}\sW$ is a web of curves  on $A$.  \begin{itemize}
      \item[(i)] Define a decomposition $f^*\sW = {\rm Fib}(g) \cup {\rm Hor}(g)$ with no common components as follows.  By Definition \ref{d.web} (iv), the morphism $g$ gives the univalent web ${\rm Fib}(g)$ on $A$,  which is a subweb
of $f^{*} \sW$. Write ${\rm Hor}(g)$ for the union of components of $f^*\sW$ different from ${\rm Fib}(g)$, i.e., those components whose general members are `horizontal' with respect to $g$.
 \item[(ii)]  Because general members of components of ${\rm Hor}(g)$ are not contracted by $g: A\to \sV$, we can
apply the construction in Definition \ref{d.pushpull} (i) to obtain a covering family of curves $g_*{\rm Hor}(g)$ on $\sV$ and the induced rational map $${\rm univ}_{g_*{\rm Hor}(g)}: {\rm Univ}_{{\rm Hor}(g)} \dasharrow {\rm Univ}_{g_* {\rm Hor}(g)}.$$  \item[(iii)] Define a decomposition $ {\rm Hor}(g) =  {\rm Inf}(g) \cup {\rm Fin}(g)$ with no common components as follows. An irreducible component $\sH$ of ${\rm Hor}(g)$ belongs to ${\rm Fin}(g)$ (resp. ${\rm Inf}(g)$) if the restriction of ${\rm univ}_{g_*{\rm Hor}(g)}$ in (ii)  to ${\rm Univ}_{\sH}$ is generically finite (resp. not generically finite) over its image in ${\rm Univ}_{g_*{\rm Hor}(g)}$. In other words, $${\rm Univ}_{{\rm Hor}(g)} \stackrel{{\rm univ}_{g_*{\rm Hor}(g)}}{\dasharrow}  {\rm Univ}_{g_*{\rm Hor}(g)}$$ is generically finite on ${\rm Fin}(g)$ and has positive-dimensional fibers on ${\rm Inf}(g)$.  \end{itemize}
\end{definition}

Next proposition shows that members of ${\rm Fin}(g)$ and ${\rm Inf}(g)$ can be distinguished by a local property of the web-structure.

\begin{proposition}\label{p.integrable}
In the setting of Definition \ref{d.ram},  let $A_{\rm reg} \subset A$ be the Zariski open subset with respect to the web $f^*\sW$ on $A$ as in Proposition \ref{p.regular}.  For $y \in A_{\rm reg}$, let
$W_i, 1 \leq i \leq d,$ be the web-structure on ${\rm Reg}(y)$ induced by $f^{*}\sW$. Assume that leaves of $W_1$ belong to  the irreducible web ${\rm Fib}(g)$  and the leaves of  $W_2$ belong to an irreducible component $\sH$ of ${\rm Hor}(g)$.
  Then $\sH \subset {\rm Inf}(g)$ if and only if the distribution $W_1 + W_2$ on ${\rm Reg}(y)$
is integrable.   \end{proposition}

\begin{proof}
We can find a member $C \subset A$ of $\sH$ through a general point  $x \in {\rm Reg}(y)$ such that  $C \cap {\rm Reg}(y)$  is the leaf of $W_2$ through $x$.  Let $F$ be the leaf of $W_1$ through $x$, which is an open subset in
$g^{-1}(g(x))$.

If $\sH$ is a component of ${\rm Inf}(g)$, then deformations of $C$ intersecting  $F$, say, $$\{ C_t, t \in \Delta, C_0 =C \}$$
 are all sent by $g: A \to \sV$ to the same curve in $\sV$, i.e., $g(C_t) = g(C)$. Thus the germ of  $g^{-1}(g(C))$ at $x$ is  the integral surface of the distribution $W_1 + W_2$.

 Conversely, assume that $W_1 + W_2$ is integrable and let
$S \subset {\rm Reg}(y)$ be its 2-dimensional leaf through $x$. Then the germs of $g(S)$ and $g(C)$ at $g(x)$ coincide and  the members of $\sH$ whose intersections with ${\rm Reg}(y)$ are contained in $S$ must be sent to the curve $g(C)$ in $\sV$. It follows that ${\rm univ}_{g_* {\rm Hor}(g)}$ is not generically finite on
 ${\rm Univ}_{\sH}$, that is, $\sH \subset {\rm Inf}(g)$.  \end{proof}

\begin{corollary}\label{c.nonintegrable}
A web of  curves $\sW$ on a projective variety $X$ is pairwise non-integrable if and only if ${\rm Fin}(g)$ is not empty for any choice of $\sV$
in Definition \ref{d.ram}.
\end{corollary}

\begin{proposition}\label{p.Fin}
Let $\sW$ on $X$ and  $\sW'$ on $X'$ be  webs of curves on  projective varieties.
 Fix an irreducible component $\sV$  of $\sW$ (resp.  $\sV'$ of $\sW'$) and denote by $g: A\to \sV$ and $f: A\to X$ (resp. $g': A' \to \sV'$ and $f': A' \to X'$) the universal family morphisms as in Definition \ref{d.ram}.  We have the webs ${\rm Fib}(g)$ and ${\rm Fin}(g)$  on $A$  (resp. ${\rm Fib}(g')$ and ${\rm Fin}(g')$ on $A'$).  Let $\varphi: U \to U'$ be a biholomorphic map between connected open subsets $U \subset A$
  and $U' \subset A'$, which sends $f^* \sW$ into $(f')^*\sW'$. If $\varphi$ sends ${\rm Fib}(g)$ into  ${\rm Fib}(g'),$ then it sends ${\rm Fin}(g)$ into ${\rm Fin}(g')$. \end{proposition}

\begin{proof}
Since the problem is local, we  may assume, by shrinking $U$ and $U'$,  that we have web-structures   $$W_1, \ldots, W_d \subset T(U) \mbox{ and } W'_1, \ldots, W'_d \subset T(U')$$ such that $W_1$ corresponds to ${\rm Fib}(g)$ (resp. $W'_1$ corresponds to ${\rm Fib}(g')$). By assumption, the differential ${\rm d} \varphi$ sends $W_1$ to $W'_1$ and $W_i$ to some $W'_j$.
  By Proposition \ref{p.integrable}, among $W_i $ (resp. $W'_i$), $1 \leq i \leq d$, those corresponding to ${\rm Fin}(g)$ (resp. ${\rm Fin}(g')$) are characterized by the property that
$W_1 + W_i$ (resp. $W'_1 + W'_i$) is not integrable. Since the integrability of such a distribution  is preserved by ${\rm d} \varphi$, we see that
$\varphi$ sends ${\rm Fin}(g)$ into $ {\rm Fin}(g')$. \end{proof}

\section{Bracket-generating webs}\label{s.B}

\begin{definition}\label{d.bracket} Let $U$ be a complex manifold and let $D \subset T(U)$ be a distribution (Pfaffian system) on $U$, i.e., a vector subbundle of the tangent bundle of $U$. By the holomorphic Frobenius theorem (applied to the Pfaffian systems defined on Zariski open subsets of $U$ generated by successive brackets of $D$),  there exists a Zariski open subset ${\rm dom}(D)\subset U$ and a holomorphic foliation, i.e., an integrable Pfaffian system, $${\rm Fol}^D \subset T({\rm dom}(D)),$$ called the {\em foliation generated by } $D$,  such that for a germ of complex submanifold $M \subset {\rm dom}(D)$ if $D_x \subset T_x(M) \subset T_x(U)$ for each $x \in M$, then the germ of the leaf of ${\rm Fol}^D$ through a point $x \in M$ is contained in $M$. We say that $D$ is {\em bracket-generating} if ${\rm Fol}^D = T({\rm dom}(D))$.  If, furthermore,  $U$ is a Zariski open subset in a projective variety $X$ and $D$ is an algebraic subbundle of $T(U)$, then ${\rm dom}(D)$ is a Zariski open subset in $X$. \end{definition}

\begin{definition}\label{d.B}
Given a web-structure $W_1, \ldots, W_d$ on a complex manifold $U,$ the linear span  $W_1 + \cdots + W_d$ gives a vector subbundle $D \subset T(U_o)$ on some Zariski open subset $U_o \subset U$. We say that the web-structure is {\em bracket-generating} if $D$ is bracket-generating.  Given a web $\sW$ of curves on a projective variety $X$, we have a Zariski open subset $X_o \subset X_{\rm reg}$ and an algebraic vector subbundle $D^{\sW} \subset T(X_o)$ spanned by the web-structures. Patching the data in Definition \ref{d.bracket}, we have a Zariski open subset ${\rm dom}(D^{\sW}) \subset X_o$ and a holomorphic foliation $${\rm Fol}^{\sW}:= {\rm Fol}^{D^{\sW}} \subset T( {\rm dom}(D^{\sW}))$$  generated by $D^{\sW}$. We say that $\sW$ is {\em bracket-generating}   if the web-structure in ${\rm Reg}(x)$ for each $x \in X_{\rm reg}$ in
Proposition \ref{p.regular} is bracket-generating. This is equivalent to saying that ${\rm Fol}^{\sW} = T({\rm dom}(D^{\sW}))$.
\end{definition}

\begin{definition}\label{d.sat}
Let $X$ be a projective variety with a web $\sW$ of curves.
For each $x \in X_{\rm reg}$, let  $I_x \subset X$ be the 1-dimensional closed algebraic subset defined by the  union of all members of $\sW$ passing through $x$. We can choose a Zariski open subset $T \subset {\rm dom}(D^{\sW}) \subset X_{\rm reg}$ such that for the incidence relation $I \subset T \times X$ defined by $$I = \{ (t,x), \ x \in I_t\},$$ the projection ${\rm pr}_T: I \to T$ is flat. For a projective subvariety $Z \subset X$, define $$I_Z := \mbox{ closure of } \bigcup_{s \in Z \cap T} I_s.$$ Note that $I_Z = \emptyset$ if $Z \cap T = \emptyset$ and  $I_Z = I_x$ when $Z$ is one point $x \in T$. Although $I_Z$ may not be irreducible,  each component of $I_Z$ contains $Z$ by the flatness of ${\rm pr}_T: I \to T$. Thus if $Z \cap T \neq \emptyset$, either $I_Z = Z$ or  every component of $I_Z$ has dimension equal to  $\dim Z+1$. We say that $Z$ is {\em saturated} (with respect to $\sW$) if $I_Z = Z$. \end{definition}

The following proposition shows that Definition \ref{d.B} can be interpreted in terms of algebro-geometric property of the web.

\begin{proposition}\label{p.sat}
In the setting of Definition \ref{d.sat}, for each $x \in T$, there exists a minimal saturated subvariety $S_x \subset X$ through $x$, in the sense that any saturated subvariety through $x$ contains $S_x$. Furthermore, the intersection $S_x \cap T$ is exactly the leaf of ${\rm Fol}^{\sW}|_T$ through $x$. It follows that $\sW$ is bracket-generating if and only if $X$ itself is the only saturated subvariety passing through a general point of $X$.
 \end{proposition}

\begin{proof}
For each $x \in T$, we define a projective subvariety $S_x$ containing $x$ in the following manner. Let $S^1_x$ be a component of $I_x$ and inductively define $S_x^{i+1}$ to be a component of $I_{S^i_x}$.
Recall from Definition \ref{d.sat} that  for a projective variety $Z \subset X, Z \cap T \neq \emptyset$, either  $I_Z = Z$ or every component of $I_Z$ has dimension equal to  $\dim Z+1$. Thus $\dim S^{i+1}_x = \dim S^i_x +1$ or $S^i_x = S^{i+1}_x$. We conclude that $S^n_x = S^{n+1}_x = \cdots,$ where $n= \dim X$.  Define $S_x := S^{n}_x$.  This is  a saturated subvariety.  Note that  if $Z \subset X$ is saturated, then $I_{Z'} \subset Z$ for any subvariety $Z' \subset Z$. Thus if $x \in Z$, then $S^i_x \subset Z$ inductively for all $i$ and $S_x \subset Z$. This proves that $S_x$ is the minimal saturated subvariety through $x$.

Let  $S \subset T$ be a leaf of the foliation ${\rm Fol}^{\sW}|_T$, an immersed complex submanifold.
If $y \in S$, then $I_y \cap T \subset S$ because each irreducible component of $I_y$ is an integral curve of $D^{\sW}$.
Thus if a projective variety $Z \subset X$ satisfies $Z \cap T \subset S$, then $I_Z \cap T \subset S$.
Denoting by $S(x)$ the leaf of ${\rm Fol}^{\sW}|_T$ through a point $x \in T,$ we have  $S^i_x \cap T \subset S(x)$ inductively for all $i$. This implies that  $S_x \cap T \subset S(x)$ and $$\dim S_x \leq {\rm rank} {\rm Fol}^{\sW}.$$
On the other hand,
 if $Z \subset X$ is saturated,  then for each nonsingular point $x \in Z \cap T$, we have $D^{\sW}_x \subset T_x(Z)$. From Definition \ref{d.bracket}, we see that $S(x) \subset Z$. Applying this to the saturated variety $Z= S_y$ for $y \in T$,  we have $S(x) \subset S_y$ for each nonsingular point $x \in S_y$. As $\dim S_y \leq \dim S(x)$,  we conclude that $S_x \cap T$ is the leaf $S(x)$ of ${\rm Fol}^{\sW}|_T$ through $x$. \end{proof}

To exploit the bracket-generating property, we need the following geometric constructions.

\begin{definition}\label{d.construct}
Let $\sV$ be an irreducible web on a projective variety ${\bf X}$. To simplify the notation,  denote by $A:= {\rm Univ}_{\sV}$ the universal family and by
$f: A \to {\bf X}$ (resp. $g: A \to \sV$) the morphism $\mu_{\sV}$ (resp. $\rho_{\sV}$). Let $\sE$ be an irreducible web  on $A$. Let $\eta: B \to A$ be a generically submersive holomorphic map from a normal analytic variety $B$. Then we can construct the following objects.
\begin{itemize}
\item[(1)]  Let $$ \begin{array}{ccc} Z & \stackrel{\alpha}{\longrightarrow} & A \\
     \varrho \downarrow & & \downarrow g \\ B & \stackrel{g \circ \eta}{\longrightarrow} & \sV \end{array} $$ be the normalization of the unique irreducible component of the pull-back $(g \circ \eta)^* A$ dominant over $B$.
           There is a canonical section $\sigma: B \to Z$  induced by $\eta$.
\item[(2)] Let $$ \begin{array}{ccc} Y & \stackrel{\widetilde{\alpha}}{\longrightarrow} & {\rm Univ}_{\sE}\\
     \nu \downarrow & & \downarrow \mu_{\sE} \\ Z & \stackrel{\alpha}{\longrightarrow} & A \end{array} $$  be the normalization of an (not necessarily unique) irreducible component  of the pull-back $\alpha^*{\rm Univ}_{\sE}$ dominant over $Z$. \end{itemize}
     We call $$ \begin{array}{ccccc} Y & \stackrel{\nu}{\longrightarrow} & Z & \stackrel{\varrho}{\longrightarrow} & B \\
     \widetilde{\alpha} \downarrow & & \downarrow \alpha & & \downarrow g \circ \eta \\ {\rm Univ}_{\sE} & \stackrel{\mu_{\sE}}{\longrightarrow} & A & \stackrel{g}{\longrightarrow} & \sV. \end{array} $$  a $(\sV, \sE)$-{\em construction} on $\eta: B \to A$. It is not uniquely determined by $(\sV, \sE),$ because the choice of $\nu$ in (2) is not unique.

     We will apply this construction to $\eta: B \to A$ obtained in the following special way.
     Let ${\bf Y}$ be an analytic variety and let $\chi: {\bf  Y} \to {\bf X}$ be a generically submersive holomorphic map.
     Let $$ \begin{array}{ccc} B & \stackrel{\eta}{\longrightarrow} & A\\
     \zeta \downarrow & & \downarrow f \\ {\bf Y} & \stackrel{\chi}{\longrightarrow} & {\bf X}\end{array} $$ be the normalization of an irreducible component (not necessarily unique) of the pull-back $\chi^*A$ dominant over ${\bf Y}$.
    A $(\sV, \sE)$-construction on $\eta: B \to A$ arising this way, with the notation $\widetilde{\bf X} = {\rm Univ}_{\sE},
    \widetilde{\bf Y} = Y, \widetilde{\chi} = \widetilde{\alpha}$ and $h = \mu_{\sE},$
    $$ \begin{array}{ccccccc} \widetilde{\bf Y} & \stackrel{\nu}{\longrightarrow} & Z & \stackrel{\varrho}{\longrightarrow} & B & \stackrel{\zeta}{\longrightarrow} & {\bf Y}\\
     \widetilde{\chi} \downarrow & & \downarrow \alpha & & \downarrow g \circ \eta & & \downarrow \chi \\\widetilde{\bf X} & \stackrel{h}{\longrightarrow} & A & \stackrel{g}{\longrightarrow} & \sV & & {\bf X}, \end{array} $$ will be called a $(\sV, \sE)$-{\em tower on} $ \chi: {\bf Y} \to {\bf X}$.
     \end{definition}

\begin{definition}\label{d.tower}
Let $\sW$ be a web on  a projective variety $X_0$ of dimension $n$.
Let $Y_0 \subset X_0$ be a connected open subset and denote by $\chi_0: Y_0 \to X_0$ the inclusion.
We will define inductively \begin{itemize} \item[(i)] a collection of projective varieties $X_1, \ldots, X_n$ with a generically finite surjective morphism $\lambda_i: X_i \to X_0$ for each $1 \leq i \leq n$; \item[(ii)]  a collection of analytic varieties $Y_1, \ldots, Y_n$ with a projective morphism $\theta_i: Y_i \to Y_0$ of relative dimension $i$ for each $1 \leq i \leq n$; and
 \item[(iii)] a generically submersive holomorphic map $\chi_i: Y_i \to X_i$ for each $1 \leq i \leq n$ \end{itemize} in the following way.

Pick an irreducible component $\sV_1$ of $\sW$ with the universal family $\sV_1 \stackrel{g_1}{\leftarrow} A_1 \stackrel{f_1}{\to} X_0$ and an irreducible component $\sE_1$ of $f_1^* \sW$. Choose a $(\sV_1, \sE_1)$-tower on $\chi_0: Y_0 \to X_0,$ to be denoted by
 $$ \begin{array}{ccccccc} Y_1 & \stackrel{\nu_1}{\longrightarrow} & Z_1 & \stackrel{\varrho_1}{\longrightarrow} & B_1 & \stackrel{\zeta_1}{\longrightarrow} & Y_0\\
 \chi_1 \downarrow & & \downarrow \alpha_1 & & \downarrow g_1 \circ \eta_1 & & \downarrow \chi_0\\X_1 & \stackrel{h_1}{\longrightarrow} & A_1 & \stackrel{g_1}{\longrightarrow} & \sV_1 & & X_0. \end{array} $$
  Denote by
       $\theta_1: Y_1 \to Y_0$ the composition $$ Y_1 \stackrel{\nu_1}{\longrightarrow} Z_1 \stackrel{\varrho_1}{\longrightarrow} B_1 \stackrel{\zeta_1}{\longrightarrow} Y_0.$$
       Then $\theta_1$ is a projective morphism with relative dimension 1. Denote by
       $\lambda_1: X_1 \to X_0$ the composition $$X_1 \stackrel{h_1}{\longrightarrow} A_1 \stackrel{f_1}{\longrightarrow} X_0.$$ Then $\lambda_1$ is
       a generically finite morphism between two projective varieties.

      To use an induction, assume that we have defined $\lambda_{i-1}: X_{i-1} \to X_0$, $\theta_{i-1}: Y_{i-1} \to Y_0$ and $\chi_{i-1}: Y_{i-1}\to X_{i-1}$ for $1 < i \leq n$.
      Pick an irreducible component $\sV_{i}$ of $\lambda_{i-1}^* \sW$
      with the universal family $\sV_{i} \stackrel{g_i}{\leftarrow} A_i \stackrel{f_i}{\to} X_{i-1}$ and an irreducible component $\sE_i$ of $(\lambda_{i-1} \circ f_i)^* \sW$. Choose a $(\sV_i, \sE_i)$-tower on $\chi_{i-1}: Y_{i-1} \to X_{i-1}$, to be denoted by
 $$ \begin{array}{ccccccc} Y_i & \stackrel{\nu_i}{\longrightarrow} & Z_i & \stackrel{\varrho_i}{\longrightarrow} & B_i & \stackrel{\zeta_i}{\longrightarrow} & Y_{i-1}\\
 \chi_i \downarrow & & \downarrow \alpha_i & & \downarrow g_i \circ \eta_i & & \downarrow \chi_{i-1}\\X_i & \stackrel{h_i}{\longrightarrow} & A_i & \stackrel{g_i}{\longrightarrow} & \sV_i & & X_{i-1}. \end{array} $$
  Denote by
       $\theta_i: Y_i \to Y_0$ the composition $$ Y_i \stackrel{\nu_i}{\longrightarrow} Z_i \stackrel{\varrho_i}{\longrightarrow} B_i \stackrel{\zeta_i}{\longrightarrow} Y_{i-1} \stackrel{\theta_{i-1}}{\longrightarrow} Y_0.$$
       Then $\theta_i$ is a projective morphism with relative dimension $i$. Denote by
       $\lambda_i: X_i \to X_0$ the composition $$X_i \stackrel{h_i}{\longrightarrow} A_i \stackrel{f_i}{\longrightarrow} X_{i-1}
       \stackrel{\lambda_{i-1}}{\longrightarrow} X_0.$$ Then $\lambda_i$ is
       a generically finite morphism between two projective varieties.

       The above inductive construction of objects (i), (ii), (iii)  depends  on the choice of $\sV_i, \sE_i$ and a choice of
       a $(\sV_i, \sE_i)$-tower at each step. We will call this collection of objects {\em a tower} on $Y_0 \subset X_0$ {\em constructed via}  $(\sV_1,  \ldots, \sV_n)$ and $(\sE_1,  \ldots, \sE_n)$.
        \end{definition}

A tower is equipped with the following special subsets $D_i \subset Y_i, 1\leq i \leq n$, which may be viewed as
the `diagonals' of the construction.

\begin{lemma}\label{l.diagonal}
For a tower on $Y_0 \subset X_0$ constructed in Definition \ref{d.tower}, define inductively a closed analytic subset
$D_i \subset Y_i$ for each $1 \leq i \leq n$ by setting
 $D_1 := \nu_1^{-1}(\sigma_1(B_1))$ and
$$D_i := \nu_i^{-1}\left( \sigma_i (\zeta_i^{-1}(D_{i-1})) \right) \mbox{ for
    each } 1 < i \leq n, $$
where $\sigma_i: B_i \to Z_i$ is the natural section of $\varrho_i$ induced by $\eta_i: B_i \to A_i$. Let $\xi_i: Y_i \to X_0$ be the composition $\lambda_i \circ \chi_i$. Then
for each $1 \leq i \leq n$, we have \begin{itemize} \item[(1)] $\theta_i(D_i) = Y_0$ and
\item[(2)] $\theta_i(y_i) = \xi_i (y_i)$ for any $y_i \in D_i$. \end{itemize} \end{lemma}

\begin{proof}
It is clear that $\theta_1(D_1) = Y_0$. Let us assume that $\theta_{i-1}(D_{i-1}) = Y_0$ for
$1 < i \leq n$.
Since $$
\varrho_i (\nu_i(D_i)) = \varrho_i \left( \sigma_i(\zeta_i^{-1}(D_{i-1})) \right) = \zeta_i^{-1}(D_{i-1}),$$ we have $$\theta_i(D_i)  = \theta_{i-1} \left( \zeta_i ( \varrho_i (\nu_i (D_i))) \right) = \theta_{i-1}( D_{i-1}) = Y_0.$$ This proves (1).

To prove (2), note that for any $y_i \in D_i, 1 \leq i \leq n,$ we have elements
$b_i \in B_i$ and $y_{i-1} \in D_{i-1}$ (with the convention $D_0 = Y_0$) such that
 $$\nu_i(y_i) = \sigma_i (b_i) \mbox{ and } \zeta_i(b_i) = y_{i-1}.$$
 Then $$\theta_1(y_1) = \zeta_1 \left( \varrho_1 ( \nu_1(y_1)) \right) = \zeta_1(b_1) = y_0,$$ while \begin{eqnarray*} \xi_1(y_1) &=& \lambda_1 \left( \chi_1 (y_1) \right) \\
  &=& (f_1 \circ h_1) \circ \chi_1 (y_1) \\ &=& f_1 \left( \alpha_1 \circ \nu_1 (y_1) \right) \\ &=& f_1 (\eta_1(b_1)) \\ &=& \chi_0(\zeta_1 (b_1)) = y_0, \end{eqnarray*}
  checking (2) for $i=1$. Now assume $\theta_{i-1}(y_{i-1}) = \xi_{i-1} (y_{i-1})$ for any $y_{i-1} \in D_{i-1}$. Then $$\theta_i(y_i) = \theta_{i-1} \left( \zeta_i ( \varrho_i ( \nu_i(y_i))) \right) = \theta_{i-1}(\zeta_i(b_i)) = \theta_{i-1}(y_{i-1}),$$ while
  \begin{eqnarray*} \xi_i(y_i) & = & \lambda_i \circ \chi_i (y_i) \\ &=& \lambda_{i-1} \circ (f_i \circ h_i) \circ \chi_i (y_i) \\ & =& \lambda_{i-1} \left( f_i ( \alpha_i (\nu_i  (y_i) ))\right) \\ & =& \lambda_{i-1} \left( f_i (\eta_i(b_i)) \right) \\&=&
  \lambda_{i-1} \left( \chi_{i-1} (\zeta_i (b_i)) \right) \\ &=&
  \lambda_{i-1} \left( \chi_{i-1} (y_{i-1}) \right) = \xi_{i-1} (y_{i-1}).\end{eqnarray*}
   Thus (2) holds by induction.
\end{proof}

The following is a key property of bracket-generating webs.

\begin{proposition}\label{p.tower}
Let $\sW$ be a bracket-generating web on  a projective variety $X_0$ of dimension $n$ and let $Y_0 \subset X_0$ be a connected open subset.
Then we can choose $(\sV_1, \ldots, \sV_n)$  in Definition \ref{d.tower}   such that for a tower on $Y_0 \subset X_0$ constructed   via $(\sV_1, \ldots, \sV_n)$ and $(\sE_1, \ldots, \sE_n),$ under any choice of $(\sE_1, \ldots, \sE_n)$,
   the morphism
$\xi_i= \lambda_i \circ \chi_i: Y_i \to X_0$
sends each irreducible
component of $\theta_i^{-1}(y)$ for a general $y \in Y_0$ to a projective subvariety of dimension $i$ in $X$. In particular, the restriction of $\xi_n$ on any irreducible component of $\theta_n^{-1}(y)$
is a generically finite morphism surjective over $X_0$.
\end{proposition}

  \begin{proof} We will choose  $(\sV_1, \ldots, \sV_n)$ inductively such that for each $i$ and a general $y \in Y_0$, $$\xi_i|_{\theta_i^{-1}(y)}: \theta_i^{-1}(y)  \to X_0$$ sends each irreducible
component of $\theta_i^{-1}(y)$ to a variety of dimension $i$.
  It is clear that we may choose $\sV_1$ as any irreducible component of $\sW$, to make this work for $i=1$.
Assuming that we have chosen $\sV_1, \ldots, \sV_i, i < n,$ satisfying the requirement, we will choose $\sV_{i+1}$ as follows.

Let $G_i \subset X_i$ be the image of an irreducible component of $\theta_i^{-1}(y)$  under $\chi_i$. Our  assumption is  $\dim G_i = i < n.$  Let $T_i \subset  X_i$ be the Zariski open subset determined by the web  $\lambda_i^*\sW$ as in  Definition  \ref{d.sat} (by substituting  $X_i$ for $X$ and $\lambda_i^* \sW$ for $\sW$).
From the generic submersiveness of $\chi_i$, the projective variety $G_i$ intersects $T_i$. Since  $\lambda_i^*\sW$ is bracket-generating on $X_i$,
  Proposition \ref{p.sat} says that $G_i$ is not saturated.   Thus we can choose an irreducible component $\sV_{i+1}$ of $\lambda_i^*\sW$ and an irreducible component $J_{i+1}$ of $f_{i+1}^{-1}(G_i)$ which is dominant over $G_i$ such that
$$ \dim g_{i+1}(J_{i+1}) = \dim J_{i+1} = i $$  and,  denoting by $\widetilde{J}_{i+1}$ the irreducible component of $ g^{-1}_{i+1} \left( g_{i+1} (J_{i+1}) \right) $ dominant over $g_{i+1}(J_{i+1}),$
$$ \dim \widetilde{J}_{i+1} = \dim f_{i+1} (\widetilde{J}_{i+1}) = i+1.$$ Choose any component $\sE_{i+1}$ of $(\lambda_i \circ f_{i+1})^*\sW$.
When $\chi_{i+1}: Y_{i+1} \to X_{i+1}$ is a $(\sV_{i+1}, \sE_{i+1})$-tower on  $\chi_i: Y_i \to X_i$,  the image of $B_{i+1}$ in $A_{i+1}$ contains $J_{i+1}$. Then  $\chi_{i+1}(\theta_{i+1}^{-1}(y))$  must have dimension $i+1$ because it contains $f_{i+1}(\widetilde{J}_{j+1})$, which has dimension $i+1$.
 Thus some component of $\theta_{i+1}^{-1}(y)$ is sent by $\xi_{i+1}= \lambda_{i+1} \circ \chi_{i+1}$ to a variety of dimension $i+1$. Since $Y_i$ is irreducible and $y$ is general, this holds for every component of
 $\theta_{i+1}^{-1}(y)$, completing the proof by induction. \end{proof}

\section{Proof of Theorem \ref{t.main}}\label{s.proof}

For the proof of Theorem \ref{t.main}, it is convenient to introduce the following terms.

\begin{definition}\label{d.mm}
 Let $ Q, R, S$  be three projective varieties of the same dimension.

\begin{itemize}
\item[(1)]  A closed algebraic subset $\Gamma \subset S \times Q$ is called a {\em generically finite algebraic correspondence} from $S$ to $Q$ if  the projections ${\rm pr}_S: \Gamma \to S$ and ${\rm pr}_Q: \Gamma \to Q$ are generically finite on every component of $\Gamma$.
We will denote by $\sC(S, Q)$ the set of generically finite algebraic correspondences from $S$ to $Q$.
For convenience, for an element  $\gamma \in \sC(S,Q)$,  we will call the corresponding  $\Gamma \subset S \times Q$  the {\em graph} of $\gamma$ and write $\Gamma = {\rm Graph}(\gamma).$ By symmetry, an element $\gamma \in \sC(S, Q)$ can be viewed as an element of $\sC(Q, S)$, to be denoted by $\gamma^{-1}$.
\item[(2)] Given an element $\gamma \in \sC(S,Q)$ and a general point $x \in S$, by composing ${\rm pr}_Q$ with the  inverse images of ${\rm pr}_S|_{{\rm Graph}(\gamma)},$ we obtain a finite number  of biholomorphic maps defined  on a neighborhood $O$    of $x,$
           $$\gamma_x^i: O \to \gamma_x^i(O) \subset Q, \ 1 \leq i \leq m,$$ where $m= m(\gamma)$ is a positive integer depending on $\gamma$.  The collection $\{ \gamma_x^1, \ldots, \gamma_x^m\}$ will be called the {\em germs} of $\gamma$ at $x$.
    \item[(3)] Given $\gamma \in \sC(R, S)$ and $\beta \in \sC(S, Q)$, we denote by
    $\beta \circ \gamma \in \sC(R, Q)$ the unique element canonically determined by the property that the germs of $\beta \circ \gamma$ at a general point $x \in R$ are the compositions of
    the germs $\{ \gamma_x^1, \ldots, \gamma_x^m \}$ of $\gamma$ at $x$ and the germs of $\beta$ at the points $\gamma_x^1(x), \ldots, \gamma_x^m(x)$.
    \item[(4)] Let $U$ be an analytic variety.  Let $\varphi: U \dasharrow S$ and $\psi: U \dasharrow Q$ be
     two generically submersive meromorphic maps. For an element  $\gamma \in \sC(S, Q),$ we write $$\psi \sim \gamma \circ \varphi \mbox{  (or }
    \varphi \sim \gamma^{-1} \circ \psi), $$  if  for a general point $y \in U$, the germ $\varphi_y$ of $\varphi$ at $y$ and the germ $\psi_y$ of $\psi$ at $y$ satisfy $$\psi_y = \gamma_{\varphi(y)}^1 \circ \varphi_y$$ for some germ $\gamma_{\varphi(y)}^1$ of $\gamma$ at $\varphi(y)$.
    In this case, for any generically submersive holomorphic map  $h: V \to U$ from an analytic variety $V$, if we write $h^*\psi= \psi \circ h$ and $h^* \varphi = \varphi \circ h$, then
    it is clear that $h^* \psi \sim \gamma \circ h^* \varphi$, which may be written as
    $$h^* (\gamma \circ \varphi) \ \sim \ \gamma \circ h^* \varphi.$$
    If $\varphi$ happens to be an inclusion of $U$ as an open subset $\varphi: U \subset S$,
    then $\psi \sim \gamma \circ \varphi$ implies that  ${\rm Graph}(\psi) \subset U \times Q \subset S \times Q$ is contained in the closed algebraic subset ${\rm Graph}(\gamma) \subset S \times Q$.
    \item[(5)] For an analytic variety $U$ and meromorphic maps
    $\varphi_1: U \dasharrow R, \varphi_2: U \dasharrow S$ and $\varphi_3: U \dasharrow Q$, if we have $\gamma \in \sC(R, S)$ and $\beta \in \sC(S, Q)$ such that
    $\varphi_2 \sim \gamma \circ \varphi_1$ and $\varphi_3 \sim \beta \circ \varphi_2$, then it is clear that $$\varphi_3 \ \sim \ (\beta \circ \gamma) \circ \varphi_1.$$ \item[(6)] Let $\sW$ be a web of curves on $S$. The {\em push-forward} of $\sW$ by $\gamma \in \sC(S, Q),$ to be denoted by $\gamma_*\sW$, is the web
     $$\gamma_* \sW := ({\rm pr}_Q|_{\Gamma})_*( {\rm pr}_S|_{\Gamma})^* \sW$$ where $\Gamma = {\rm Graph}(\gamma).$ Here, we are applying the pull-back and the push-forward  to all irreducible components of $\Gamma$ and taking the union of all the resulting webs.  Note that $\sW$ is a subweb of, but not necessarily the same web as,  the web $\gamma^{-1}_*(\gamma_* \sW).$
\end{itemize}
\end{definition}

\begin{definition}\label{d.descend}
Let $\alpha: Z \to A$ be a submersive holomorphic map between analytic varieties. Let $U \subset Z$ be an open subset and $\varphi: U \dasharrow Q$ be a meromorphic map to
an analytic variety $Q$. We say that $\varphi$ is $\alpha$-{\em descending} if for each $u \in U$ there exists a germ  $\varphi_{\alpha(u)}$ of meromorphic maps from
$A$ to $Q$ at the point $\alpha(u) \in A$ such that $\varphi_{\alpha(u)} \circ \alpha$ coincides with the germ of $\varphi$ at $u$.
If furthermore, $\dim A = \dim Q$ and there are webs of curves $\sW$ on $A$ and $\sU$ on $Q$, we say that $\varphi$ {\em sends} $\sW$ into $\sU$ if $\varphi_{\alpha(u)}$ sends $\sW$ into $\sU$ at some (hence any) $u \in U$ in the sense of Definition \ref{d.respect}.
\end{definition}

     \begin{lemma}\label{l.lift}
     Let $\varrho: Z \to B$ be a proper holomorphic map between normal analytic varieties equipped with a section $\sigma: B \to Z$. Let $\alpha: Z \to A$ and $\eta: B \to A$ be generically submersive holomorphic maps to an analytic variety $A$,  satisfying
     $\eta = \alpha \circ \sigma$. Let $Q$ be an analytic variety and $ \varphi: B \dasharrow Q $ be an $\eta$-descending meromorphic map. Then there exists a neighborhood
     $N(B) \subset Z$ of $\sigma(B)$ and an $\alpha$-descending meromorphic map $\varphi^{\sigma}: N(B) \dasharrow Q$ such that $\varphi= \varphi^{\sigma} \circ \sigma.$ \end{lemma}

     \begin{proof} Since $\varphi$ is $\eta$-descending, for each $u \in B$, we have a germ $\varphi_{\eta(u)}$ of meromorphic maps at $\eta(u)$ such that the germ of $\varphi$ at $u$ is just
      $\varphi_{\eta(u)} \circ \eta$. Let $(\varphi^{\sigma})_{\sigma(u)}$ be the germ of meromorphic maps at $\sigma(u)$ given by $\varphi_{\eta(u)} \circ \alpha$. The collection of germs
      $(\varphi^{\sigma})_{\sigma(u)}$ as $u$ varies over $B$ define the meromorphic map $\varphi^{\sigma}$ in a neighborhood of $\sigma(B)$ with the desired properties. \end{proof}

      Next proposition is a crucial step of the proof of Theorem \ref{t.main}.

    \begin{proposition}\label{p.mainext}
      Let ${\bf X}$ be a projective variety with a  web $\sW$ which is pairwise non-integrable.
      Let $\sV$ be an irreducible component of $\sW$ with  the universal family morphisms $f:A \to {\bf X}$ and $g: A \to \sV$, as in Definition \ref{d.construct}.  Fix a component $\sE$ of ${\rm Fin}(g),$ which is nonempty because $\sW$ is pairwise non-integrable.
        Let $B$ be a normal  analytic variety and let $\eta: B \to A$ be a generically submersive holomorphic map.
        Let $$ \begin{array}{ccccc} Y & \stackrel{\nu}{\longrightarrow} & Z & \stackrel{\varrho}{\longrightarrow} & B \\
     \widetilde{\alpha} \downarrow & & \downarrow \alpha & & \downarrow g \circ \eta \\ {\rm Univ}_{\sE} & \stackrel{\mu_{\sE}}{\longrightarrow} & A & \stackrel{g}{\longrightarrow} & \sV. \end{array} $$
                be a $(\sV, \sE)$-construction on $\eta: B \to A$  from Definition \ref{d.construct}.
        Let $N(B)$ be a neighborhood of $\sigma(B)$ where $\sigma: B \to Z$ is the canonical section of $\varrho$ and let $\varphi: N(B) \dasharrow Q$ be an $\alpha$-descending meromorphic map to a projective variety $Q$,  which sends $f^*\sW$ into a web $\sU$ on $Q$ in the sense of Definition \ref{d.descend}.
       Then there exists a projective variety $S$ with $\gamma \in \sC(S, Q)$ and an $\widetilde{\alpha}$-descending  meromorphic map $\widetilde{\varphi}: Y \dasharrow S$ such that for any $b \in B$ and any   $x \in \nu^{-1}(\sigma(b)),$ the germ $\widetilde{\varphi}_x$ of $\widetilde{\varphi}$ at $x$,
       the germ $\varphi_{\sigma(b)}$ of $\varphi$ at $\sigma(b)$ and  the germ $\nu_x$ of $\nu$ at $x$ satisfy \begin{eqnarray}\label{e.sim} \nu_x^* \varphi_{\sigma(b)} &  \sim & \gamma \circ \widetilde{\varphi}_x, \end{eqnarray} i.e., the following diagram, where we use a double arrow to denote the algebraic correspondence $\gamma$, commutes at the level of germs.
       $$ \begin{array}{ccccccc} {\rm Univ}_{\sE} & \stackrel{\widetilde{\alpha}}{\longleftarrow} & Y & =& Y & \stackrel{\widetilde{\varphi}}{\dasharrow} & S  \\ \mu_{\sE} \downarrow & & \downarrow \nu & & & & \Downarrow \gamma \\ A & \stackrel{\alpha}{\longleftarrow} &
        Z & \supset & N(B) &  \stackrel{\varphi}{\dasharrow} & Q. \end{array}$$ In particular, the $\widetilde{\alpha}$-descending map $\widetilde{\varphi}$ sends
        the web $(f \circ \mu_{\sE})^*\sW$ into the web $(\gamma^{-1})_*\sU$ on $S$ in the sense
        of Definition \ref{d.mm} (6) and Definition \ref{d.descend}.
       \end{proposition}

\begin{proof}
By shrinking $N(B)$ if necessary, we may assume that each fiber of $\varrho$ intersects $N(B)$ along a connected set.
Since $\varphi$ sends $f^*\sW$ into $\sU$, there exists an irreducible component, say $\sV'$, of $\sU$ such that $\varphi$ sends germs of fibers of $\varrho$ at $\sigma(B)$
 to germs of members of $\sV'$, inducing a $(g \circ \eta)$-descending meromorphic map $\varphi^{\sharp}: B \dasharrow \sV'$.  Setting $ A':= {\rm Univ}_{\sV'},$ write $g': A' \to \sV'$ for $\rho_{\sV'}$ and  $f': A' \to Q$ for $\mu_{\sV'}$. Then $\varphi$ induces an  $\alpha$-descending meromorphic map $\varphi': N(B) \dasharrow A'$ such that $\varphi = f' \circ \varphi'$ and $\varphi'$ sends ${\rm Fib}(g)$ into ${\rm Fib}(g')$. It follows by Proposition \ref{p.Fin} that $\varphi'$ sends $\sE \subset {\rm Fin}(g)$ into
  ${\rm Fin} (g').$   $$ \begin{array}{ccccccc}  & & & & N(B) & \stackrel{\varphi}{\dasharrow} & Q \\ & & & &  \parallel & & \uparrow f'  \\ A & \stackrel{\alpha}{\longleftarrow} & Z & \supset & N(B) & \stackrel{\varphi'}{\dasharrow} & A'  \\ g \downarrow & & \downarrow \varrho & & \downarrow & & \downarrow g' \\ \sV & \stackrel{g \circ \eta}{\longleftarrow} &
B & = & B &  \stackrel{\varphi^{\sharp}}{\dasharrow} &\sV'. \end{array}$$
 We have the covering families of curves $\sF:= g_*\sE$ on $\sV$ and $\sF':= (g')_*{\rm Fin}(g')$ on $\sV'$ with the universal family morphisms $$ \begin{array}{ccc}
{\rm Univ}_{\sE} & \stackrel{{\rm univ}_{g_{*}\sE}}{\dashrightarrow} & {\rm Univ}_{\sF} \\
\mu_{\sE} \downarrow & & \downarrow \mu_{\sF}  \\ A & \stackrel{g}{\longrightarrow} & \sV \end{array}   \mbox{ and } \begin{array}{ccc}
{\rm Univ}_{{\rm Fin}(g')} & \stackrel{{\rm univ}_{g'_{*}{\rm Fin}(g')}}{\dashrightarrow} & {\rm Univ}_{\sF'} \\
\mu_{{\rm Fin}(g')} \downarrow & & \downarrow \mu_{\sF'}  \\ A' & \stackrel{g'}{\longrightarrow} & \sV'. \end{array} $$
where the two rational maps on the first row are  generically finite by the definition of ${\rm Fin}(g)$ and ${\rm Fin}(g')$.

 Fix a point $b \in B$ and set $u = \sigma(b) \in \sigma(B).$  By the $\alpha$-descending property of $\varphi'$, we have  a neighborhood  $\alpha (u) \in O \subset A$ with a meromorphic map $\varphi'_{\alpha(u)}: O \dasharrow A'$
 such that $\varphi'_{\alpha(u)} \circ \alpha$ gives the germ of $\varphi'$ at $u$. Similarly,  we have a neighborhood $(g \circ \eta)(b)  \in V \subset \sV$ with a meromorphic map
 $\phi: V \dasharrow \sV'$ such that $\phi \circ (g \circ \eta)$ gives the germ of $\varphi^{\sharp}$ at $b$. We may assume that $g(O) \subset V$
 by shrinking $O$ if necessary.  Let ${\bf U} \subset {\rm Univ}_{\sE}$ be a dense open subset in  ${\rm univ}_{g_*\sE} (\mu_{\sE}^{-1}(O))$. We have  $\mu_{\sF}({\bf U}) \subset V$. Since
 $\varphi'$ sends $\sE$ into ${\rm Fin}(g')$, it induces a generically biholomorphic  meromorphic map  $\psi: {\bf U}\dasharrow {\rm Univ}_{\sF'}$ such that
 the diagram
   $$ \begin{array}{ccccc} {\rm Univ}_{\sF} & \supset & {\bf U} & \stackrel{\psi}{\dasharrow} &  {\rm Univ}_{\sF'} \\
\mu_{\sF} \downarrow & &  \downarrow & & \downarrow \mu_{\sF'} \\
\sV & \supset & V & \stackrel{\phi}{\dasharrow}  & \sV' \end{array} $$ commutes and
 $$\phi \left( \mu_{\sF}(\rho_{\sF}^{-1}(\rho_{\sF}(y)) \cap {\bf U})\right) \subset \mu_{\sF'}\left(\rho_{\sF'}^{-1} (\rho_{\sF'} (\psi (y)))  \right)$$ for a general $ y \in {\bf U}.$
 Thus we can apply
 Proposition \ref{p.Psi} to obtain  a generically biholomorphic meromorphic map $$\mu_{\sF}^{-1}(V)  \stackrel{\Psi}{\dasharrow} \mu_{\sF'}^{-1} (V')
 \mbox{ satisfying } \psi = \Psi|_{\bf U}.$$ Then
 the composition $$(g \circ \mu_{\sE} \circ \widetilde{\alpha})^{-1}(V) \stackrel{\widetilde{\alpha}}{\to} (g \circ \mu_{\sE})^{-1}(V) \stackrel{{\rm univ}_{g_{*}{\rm Fin}(g)}}{\dasharrow} \mu_{\sF}^{-1}(V) \stackrel{\Psi}{\dasharrow} {\rm Univ}_{\sF'}$$ is a meromorphic map, to be denoted by $\widetilde{\varphi}_{(\varrho \circ \nu)^{-1}(b)}$,   defined in a neighborhood of $(\varrho\circ \nu)^{-1}(b)$ in $Y$. Since $\widetilde{\varphi}_{(\varrho \circ \mu)^{-1}(b)}$ is uniquely determined by the germ of $\varphi$ at $u$, the collection $\{ \widetilde{\varphi}_{(\varrho \circ \mu)^{-1}(b)}, \ b \in B \}$ defines a meromorphic map $\widetilde{\varphi}: Y \dasharrow {\rm Univ}_{\sF'}$. Let $S$ be the irreducible component of ${\rm Univ}_{\sF'}$ where the image of $\widetilde{\varphi}$ lies. Then there exists $\gamma \in \sC(S, Q)$ given by the diagram
 $$ \begin{array}{ccccccc} & &
{\rm Univ}_{{\rm Fin}(g')} & \stackrel{{\rm univ}_{g'_{*}{\rm Fin}(g')}}{\dashrightarrow} & {\rm Univ}_{\sF'} & \supset & S\\ & &
\mu_{{\rm Fin}(g')} \downarrow & & \downarrow \mu_{\sF'} & &   \\ Q & \stackrel{f'}{\longleftarrow} &  A' & \stackrel{g'}{\longrightarrow} & \sV' \end{array} $$
where the first row, the first vertical arrow and the first horizontal arrow in the third row are generically finite.
The property (\ref{e.sim}) is immediate from the construction.
 \end{proof}

 We are ready to finish the proof of Theorem \ref{t.main}.

 \begin{proof}[Proof of Theorem \ref{t.main}]
 The given conditions on $\varphi: U \to U'$ say that $\varphi$ sends $\sW$ into $\sW'$. Set $n= \dim X$, $Y_0 = U$ and $X_0=X$.
 As the web $\sW$ is bracket-generating, we can find a tower  on $Y_0 \subset X_0$ constructed via a certain choice of  $(\sV_1, \ldots, \sV_n)$ and any choice of $(\sE_1, \ldots, \sE_n),$   satisfying the property stated in Proposition \ref{p.tower}.
 Since $\sW$ is pairwise non-integrable, so is $(\lambda_{i-1} \circ f_i)^*\sW$ on $A_i$. Thus we can assume that we have chosen  $\sE_i$  as a component of
${\rm Fin}(g_i) \subset (\lambda_{i-1} \circ f_i)^*\sW$ for each $i$. Using the diagram from Definition \ref{d.construct} and Definition \ref{d.tower},
    $$ \begin{array}{ccccccc} & & & & A_i & \stackrel{f_i}{\longrightarrow} & X_{i-1} \\
    & & & & \eta_i \uparrow & & \uparrow \chi_{i-1} \\
         Y_i & \stackrel{\nu_i}{\longrightarrow} & Z_i & \stackrel{\varrho_i}{\longrightarrow} & B_i & \stackrel{\zeta_i}{\longrightarrow} & Y_{i-1}\\
     \chi_i \downarrow & & \downarrow \alpha_i & & \downarrow g_i \circ \eta_i & & \\ X_i & \stackrel{h_i}{\longrightarrow} & A_i & \stackrel{g_i}{\longrightarrow} & \sV_i  & &  \end{array} $$
 and starting with $\beta_0$ defined as the identity $X'_0 := X',$ we will find inductively  for each $1 \leq i \leq n,$ \begin{itemize}
 \item[(i)] a projective variety $X'_i$ and $\beta_i \in \sC(X'_i, X')$;
\item[(ii)]  a generically submersive meromorphic map $\varphi_i: B_i \dasharrow X_{i-1}'$ which is  $\eta_i$-descending and sends  $(\lambda_{i-1}\circ f_i)^*\sW$ into $(\beta^{-1}_{i-1})_*\sW'$; and
\item[(iii)] a generically submersive meromorphic map  $\widetilde{\varphi}_i: Y_i \dasharrow X_i'$ which is $\chi_i$-descending and sends $\lambda_i^*\sW$ into $(\beta^{-1}_i)_*\sW'$,
    \end{itemize}
    such that
     for any point  $y_i \in D_i \subset Y_i$ of Lemma \ref{l.diagonal} with $y := \xi_i(y_i) \in Y_0$, the germ $\widetilde{\varphi}_{i,y_i}$ of $\widetilde{\varphi}_i$ at $y_i$,
       the germ $\varphi_y$ of $\varphi$ at $y$ and  the germ $\xi_{i, y_i}$ of $\xi_i$ at $y_i$ satisfy \begin{eqnarray}\label{e.i}   \xi_{i, y_i}^*\varphi_y & \sim &\beta_i \circ \widetilde{\varphi}_{i,y_i},\end{eqnarray} or equivalently, the germ $\widetilde{\varphi}_{i,x_i}$ at $x_i:= \chi_i(y_i),$ representing
       $\widetilde{\varphi}_{i,y_i}$ by the $\chi_i$-descending property, satisfies \begin{eqnarray}\label{e.ii}
       \lambda_{i, x_i}^* \varphi_y & \sim & \beta_i \circ \widetilde{\varphi}_{i, x_i} \end{eqnarray}
       where $\lambda_{i, x_i}$ is the germ of $\lambda_i$ at $x_i$.

 To start with, set  $$\varphi_1 := \zeta_1^*\varphi = (f_1 \circ \eta_1)^* \varphi : B_1 \dasharrow X'$$ which is $\eta_1$-descending
 and sends $f_1^*\sW$ into $ \sW'$. For a point $b_1 \in  B_1$ with $\zeta_1(b_1) = y \in Y_0$,  the germ $\varphi_{1, \eta_1(b_1)}$ at $\eta_1(b_1) \in A_1$ representing $\varphi_1$ satisfies \begin{eqnarray}\label{e.b1} \varphi_{1,\eta_1( b_1)} &=&  f_{1, \eta_1(b_1)}^* \varphi_y.\end{eqnarray}  Applying Lemma \ref{l.lift}, we have   $$\varphi_1^{\sigma_1}: N(B_1) \dasharrow X'$$ defined in a neighborhood $N(B_1)$ of $\sigma_1(B_1)$, which is $\alpha_1$-descending.  By Proposition \ref{p.mainext}, we find
$ \gamma_1 \in \sC(X_1', X')$ and a $\chi_1$-descending meromorphic map $$\widetilde{\varphi}_1 = \widetilde{(\varphi_1^{\sigma_1})} : Y_1 \dasharrow X_1'$$ such that,  by (\ref{e.sim}) of Proposition \ref{p.mainext}, for $y_1 \in \nu_1^{-1}(\sigma(b_1))$ and $x_1 = \chi_1(y_1)$,  $$h_{1, x_1}^*  \varphi_{1, \eta_1(b_1)}  \  \sim \ \gamma_1 \circ \widetilde{\varphi}_{1,x_1}.$$ Then
 (\ref{e.b1}) and $\lambda_1 = f_1 \circ h_1$ give $$ \lambda_{1, x_1}^* \varphi_y \ \sim \  \gamma_1 \circ \widetilde{\varphi}_{1, x_1}.$$
Thus (\ref{e.ii}) holds  for $i=1$ if we put $\beta_1 = \gamma_1$.

Assume that we have found $ \beta_i \in \sC(X'_i, X'),$  $\varphi_i$ and $ \widetilde{\varphi}_i$ satisfying (\ref{e.ii}) for $1 \leq i <n$.
Define $$\varphi_{i+1} = \zeta_{i+1}^* \widetilde{\varphi}_i  : B_{i+1} \dasharrow X'_i,$$ which is $\eta_{i+1}$-descending
 by $\chi_i \circ \zeta_{i+1} = f_{i+1} \circ \eta_{i+1} $ and sends $(\lambda_i \circ f_{i+1})^*\sW$ into $(\beta^{-1}_i)_* \sW'$.
This implies that for a point $y_i \in D_i, x_i:= \chi_i(y_i)$ and $b_{i+1} \in B_{i+1}$ with $\zeta_{i+1}(b_{i+1}) = y_i$,
  the germ $\varphi_{i+1, \eta_{i+1}(b_{i+1})}$ at $\eta_{i+1}(b_{i+1}) \in A_{i+1}$ representing $\varphi_{i+1}$ satisfies \begin{eqnarray}\label{e.bi+1} \varphi_{i+1,\eta_{i+1}( b_{i+1})} &=&  f_{i+1, \eta_{i+1}(b_{i+1})}^* \widetilde{\varphi}_{x_i}.\end{eqnarray}
Applying Lemma \ref{l.lift},
 we have $$\varphi_{i+1}^{\sigma_{i+1}}: N(B_{i+1}) \dasharrow X'_i$$ defined in a neighborhood $N(B_{i+1})$ of $\sigma_{i+1}(B_{i+1})$, which is $\alpha_{i+1}$-descending.  By Proposition \ref{p.mainext}, we find $ \gamma_{i+1} \in \sC(X'_{i+1}, X'_i)$  and
$$\widetilde{\varphi}_{i+1} = \widetilde{(\varphi_{i+1}^{\sigma_{i+1}})} : Y_{i+1} \dasharrow X_{i+1}'$$ which is
$\chi_{i+1}$-descending, such that  for  $$y_{i+1} \in \nu_{i+1}^{-1}(\sigma_{i+1}(b_{i+1})) \in D_{i+1} \mbox{ and } x_{i+1} := \chi_{i+1} (y_{i+1}),$$ we have by
 (\ref{e.sim}), \begin{eqnarray}\label{e.hi+1}
  h_{i+1, x_{i+1}}^*  \varphi_{i+1, \eta_{i+1}(b_{i+1})}  &  \sim & \gamma_{i+1} \circ \widetilde{\varphi}_{i+1,x_{i+1}}. \end{eqnarray}
  Then, in the notation of Definition \ref{d.mm} (4) and (5),
    \begin{eqnarray*}
   \lambda^*_{i+1, x_{i+1}} \varphi_y & = & (f_{i+1} \circ h_{i+1})_{x_{i+1}}^* (\lambda_{i, x_i}^* \varphi_y) \mbox{ from } \lambda_{i+1} = \lambda_i \circ f_{i+1} \circ h_{i+1} \\
& \sim & \beta_i \circ \left( (f_{i+1} \circ h_{i+1})_{x_{i+1}}^* \widetilde{\varphi}_{i, x_i} \right)  \ \mbox{ by  the induction on (\ref{e.ii}) } \\ & \sim &\beta_i \circ  h_{i+1, x_{i+1}}^* \varphi_{i+1, \eta_{i+1}(b_{i+1})} \ \mbox{ by (\ref{e.bi+1})} \\
  & \sim & (\beta_i \circ \gamma_{i+1}) \circ \widetilde{\varphi}_{i+1, x_{i+1}} \ \mbox{ by (\ref{e.hi+1})}. \end{eqnarray*}
Thus  (\ref{e.ii})  holds for $i+1$  if we put $\beta_{i+1} = \beta_{i} \circ \gamma_{i+1}.$
  This completes the inductive construction of (i), (ii) and (iii) satisfying (\ref{e.i}).

Now let $F$  be a general fiber of $\theta_n : Y_n \to Y_0 = U$ such that the restriction
$\widetilde{\varphi}_n|_{Y} : F  \dasharrow X'_n$ is a well-defined rational  map, which we denote by $\Phi$.
The restriction  $\xi_n|_F: F \to X$ is generically finite on each irreducible component of $F$ by our choice of $(\sV_1, \ldots, \sV_n)$ in
Proposition \ref{p.tower}. By  (\ref{e.i}) and $\xi_n(D_n) = Y_0$ from Lemma \ref{l.diagonal}, there is an irreducible component $R$ of $F$ such that $\xi_n|_R$ regarded as
an element of $\sC(R, X)$ satisfies
$$\varphi \sim \left(\beta_n \circ \Phi|_R \circ (\xi_n|R)^{-1}\right) \circ \iota$$ for
 the inclusion $\iota: U \subset X$. As mentioned in Definition \ref{d.mm} (4), this implies that the graph of $\varphi$ is contained in the graph of a generically finite algebraic
correspondence between $X$ and $X'$.
\end{proof}

\section{\'Etale webs of smooth  curves}\label{s.etale}

\begin{definition}\label{d.etale} Let $X$ be a projective manifold, i.e., a nonsingular  projective variety.
An {\em \'etale web of smooth curves} on $X$  is a web $\sW$ of curves on $X$ with the following additional property:
 in terms of the universal family $\mu_{\sW}: {\rm Univ}_{\sW} \to X$ and $ \rho_{\sW}: {\rm Univ}_{\sW} \to \sW$, there exists a dense Zariski open subset $\sW^{\rm etale}$ of
the smooth locus of $\sW$ such that
for each $a
\in \sW^{\rm etale}$,
\begin{itemize} \item[(i)] $\rho_{\sW}^{-1}(a)$ is a smooth curve;
\item[(ii)]
$\mu_{\sW}|_{\rho_{\sW}^{-1}(O_a)}: \rho_{\sW}^{-1}(O_a) \to X$ is
unramified for some open neighborhood $O_{a}$ of $a$ in
$\sW^{\rm etale}$.
\end{itemize}
 For a point $a \in \sW^{\rm etale}$, the smooth curve
$$P_a := \mu_{\sW}(\rho_{\sW}^{-1}(a))  \ \subset \ X$$ is called a {\em regular member of the web}
$\sW$. A regular member has trivial normal bundle by (ii). Conversely, it is easy to see that a web $\sW$ of curves on $X$ is an \'etale web of smooth curves if a general member of $\sW$ is smooth and has trivial normal bundle in $X$.
When we work with an \'etale web $\sW$ of smooth curves on $X$, we will choose $X_{\rm reg}$ of Proposition \ref{p.regular} such that
$\mu_{\sW}^{-1} (X_{\rm reg}) \subset \rho_{\sW}^{-1}(\sW^{\rm etale}).$ This implies that any member of $\sW$ intersecting
$X_{\rm reg}$ is a regular member.
 If regular members of the web are rational curves, we say that $\sW$ is an {\em \'etale web of smooth rational curves}.
\end{definition}

\begin{lemma}\label{l.positive}
Let $\sW$ be an irreducible \'etale web of smooth curves on a projective manifold $X$ and let $H \subset X$ be an irreducible hypersurface which has  positive intersection number with
 members of $\sW$. Then there exist non-empty Zariski open subsets $H^{\sW} \subset H$ and $\sW^{H} \subset \sW^{\rm etale}$ such
that
 \begin{itemize}\item[(i)] for any $a \in \sW^H$,  the regular member $P_a$ intersects $H$ transversally and $P_a \cap H \subset H^{\sW}$;
 \item[(ii)] for any $x \in H^{\sW}$, there exists $a \in \sW^H$ with $x \in P_a$;   and
 \item[(iii)] if  a member $P\subset X$ of $\sW$ contains a point of $H^{\sW}$, then $P = P_a$ for some $a \in \sW^H$.
 \end{itemize}
 \end{lemma}

 \begin{proof}
 Choose a Zariski open subset $\sW^o \subset \sW^{\rm etale}$ such that $\rho_{\sW}^{-1}(\sW^o) \cap \mu_{\sW}^{-1}H$ is \'etale over $\sW^o$. Set $H^o := \mu_{\sW}(\rho_{\sW}^{-1}(\sW^o)) \cap H$.
They satisfy the required conditions. \end{proof}

The following is immediate from Definition \ref{d.etale} (ii).

\begin{lemma}\label{l.trivial}
An \'etale web of smooth curves on a projective manifold of Picard number 1 cannot be univalent.
\end{lemma}

The next lemma is proved in Proposition 6 of \cite{HM03}. It follows from the fact that the base-change of an \'etale morphism is also an \'etale morphism.

\begin{lemma}\label{l.coveretale}
Let $f: Y \to X$ be a generically finite morphism between two projective manifolds. If $\sW$ is an \'etale web of smooth curves on $X$, then $f^*\sW$ is an \'etale web of smooth curves on $Y$. \end{lemma}

\begin{proposition}\label{p.surface}
On a nonsingular projective surface, an irreducible \'etale web of smooth curves must be univalent. \end{proposition}

\begin{proof}
A regular member $C$ of  an \'etale web of smooth curves on a nonsingular projective surface satisfies $C\cdot C =0$ because $C$ has trivial normal bundle, as explained in Definition \ref{d.etale}.
If the web is not univalent, two distinct members $C$ and $ C^{'}$ through a general point satisfy $C \cdot C^{'} >0$.
It follows that $C$ and $C^{'}$ cannot belong to the same component of the web, i.e., the web cannot be irreducible.
\end{proof}

\begin{proposition}\label{p.infty}
Let $\sW$ be an \'etale web of smooth curves on a projective manifold $X$. Fix an  irreducible component $\sV$ of $\sW$ and let $f: A={\rm Univ}_{\sV} \to X$ and $g: A={\rm Univ}_{\sV} \to \sV$ be the universal family morphisms. The web $f^*\sW$ on $A$ has the subweb  ${\rm Fin}(g)$ from Definition \ref{d.ram} and the subweb ${\rm Mult}(f^*\sW)$ of $f^* \sW$ from Definition \ref{d.pushpull}.
Then  ${\rm Mult}(f^*\sW) \subset {\rm Fin}(g)$. \end{proposition}

\begin{proof}
Assuming that we have a common irreducible subweb $\sJ$ of ${\rm Inf}(g)$ and ${\rm Mult}(f)$, we will derive a contradiction.
 Let $J \subset A$ be a general member of $\sJ$ and let $W = g(J) \subset \sV$.
Since $J$ is a member of ${\rm Inf}(g)$, we have a family of distinct members $\{ J_t, t \in \Delta, J= J_0\}$ of $\sJ$ such that $W = g(J_t)$ for all $t\in \Delta$. Write $W^o = W \cap \sW^{\rm etale}.$ Let $S \subset X$ be the closure of $f(g^{-1}(W^o)).$
Then $W^o$ determines a web of curves on $S$ whose general members are
given by the two morphisms $\rho_{W^o}: g^{-1}(W^o) \to W^o$ and $\mu_{W^o}: g^{-1}(W^o) \to S$ obtained by the restrictions of $g$ and $f$. Choose a desingularization $\nu: \widetilde{S} \to S$. Then the property (ii) of Definition \ref{d.etale} is preserved under the pull-back of $\mu_{W^o}$ by $\nu$. It follows that $\rho_{W^o}$ and $\mu_{W^o}$ give rise to an irreducible \'etale web of smooth curves on $\widetilde{S}$.
By Proposition \ref{p.surface}, this web must be univalent.
This means that the morphism $\mu_{W^o}: g^{-1}(W^o) \to S$ is birational.
 But the closure of the surface $g^{-1}(W^o)$
is covered by the members $\{ J_t, t \in \Delta\}$ of $\sJ$. Since the restriction of $f$ to each $J_t$ is not birational by the assumption
 $\sJ \subset {\rm Mult}(f)$, the morphism $\mu_{W^o}$ cannot be birational,  a contradiction. \end{proof}

\begin{proposition}\label{l.nbirational}
Let $X$ be a simply connected projective manifold of Picard number 1 equipped with an \'etale web $\sW$ of smooth curves.
Let $M$ be a projective variety and let $p: M \to X$ be a generically finite morphism which is not birational.  Then
${\rm Mult}(p^*\sW)$  is not empty. \end{proposition}

\begin{proof}
We may assume that $M$ is nonsingular by taking desingularization.
Let $H \subset X$ be an irreducible component of the reduced branch divisor of $p$, which is nonempty by the assumptions that $X$ is simply connected and $p$ is not birational.
Let $R \subset M$ be an irreducible component of the ramification divisor of $p$ such that $p(R) = H$. Since $H$ is ample, there exists an irreducible component $\sV$ of  $p^*\sW$ whose members have positive intersection with $R$. Since $\sV$ is an \'etale web of smooth curves by Lemma \ref{l.coveretale}, a general member $C$ of $\sV$ intersects $R$ transversally by Lemma \ref{l.positive} at a nonsingular point $y$ of $R$. As $p(C)$ is a general member of $p_*\sV \subset \sW$,  it intersects $H$ transversally at a nonsingular point $p(y) \in H$, by Lemma \ref{l.positive} again.   This implies that $p|_C: C \to p(C)$ is ramified at $y$, hence is not birational.  It follows that $\sV \subset {\rm Mult}(p^*\sW)$. \end{proof}

\begin{proposition}\label{p.NBir}
Let $\sW$ be an \'etale web of smooth  curves on a simply connected projective manifold $X$  of Picard number 1.
Fix an  irreducible component $\sV$ of $\sW$ and let $f: {\rm Univ}_{\sV} \to X$ and $g: {\rm Univ}_{\sV} \to \sV$ be the universal family morphisms.
Then ${\rm Mult}(f^*\sW) \neq \emptyset$ and the web $\sW$ is pairwise non-integrable.
 \end{proposition}

 \begin{proof}
 Since $X$ is nonsingular of Picard number 1, it cannot have a univalent \'etale web of smooth curves.
 It follows that the morphism $f: {\rm Univ}_{\sV} \to X$ is not birational.
By Proposition \ref{l.nbirational}, we see that ${\rm Mult}(f^* \sW) \neq \emptyset$.  Then Proposition \ref{p.infty} shows ${\rm Fin}(g) \neq \emptyset$, hence $\sW$ is pairwise non-integrable by Corollary \ref{c.nonintegrable}.
 \end{proof}

The proof of the following proposition is essentially the same as that of  Lemma 3.1 of \cite{HM01}.

\begin{proposition}\label{p.chain}
Any \'etale web of smooth curves on a projective manifold of Picard number 1 is bracket-generating. \end{proposition}

\begin{proof}
Suppose that there exists a projective manifold $X$ of Picard number 1 with an \'etale web $\sW$ of smooth curves which is not bracket-generating. By Proposition \ref{p.sat},   there exists an irreducible subvariety $\sH$ of the Hilbert scheme of $X$ whose general member is a  saturated  subvariety of $X$ of dimension strictly smaller than $X$ and whose members cover the whole $X$. Then by choosing a suitable subvariety of $\sH$, we obtain a hypersurface $H \subset X$ which is the closure of the union of some collection of saturated subvarieties of $X$. Since $X$ is of Picard number 1, members of each irreducible component $\sV$ of $\sW$ have positive intersection number with $H$. From Lemma \ref{l.positive}, we have a Zariski open subset $H^{\sV} \subset H$, such that for any $b \in H^{\sV}$, we have a member $P_a$ of $\sV$ intersecting $H$ transversally at $b$.
By our construction of $H$, we have a saturated subvariety $S \subset H$ with $ S \cap H^{\sV} \neq \emptyset.$ Pick a general
point $b \in S$. Then we have $P_a$ with   $b \in P_a  \cap S$ and $P_a \not\subset S$, which means that $S$ is not saturated, a contradiction.
\end{proof}

The following is well-known. We will give a proof  for the reader's convenience.

\begin{proposition}\label{p.example}
Let $\ell >0$ be a fixed integer.
Let $X \subset \BP^N$ be a projective submanifold such that there are nonempty, but only finitely many smooth rational curves of degree $\ell$   through a general point of $X$.
 Let $\sW$ be the web of curves on $X$ such that  members of $\sW$ through a general point $x \in X$ are exactly  smooth rational curves of degree $\ell$ on $X$ through $x$.  Then $\sW$ is an \'etale web of smooth  curves on $X$. If, furthermore, the Picard number of $X$ is 1, then $X$ is a Fano manifold and any \'etale web of smooth curves on $X$ is a subweb of $\sW$.   \end{proposition}

\begin{proof}
It is well-known (see Chapter II of \cite{Ko}) that the normal bundle of a smooth rational curve $C$ of fixed degree $\ell$ through a general point $x$ of the projective manifold $X$ is semi-positive, i.e.,
 $$ N_{C \subset X} \cong \sO(a_1) \oplus \cdots \oplus \sO(a_{n-1}), \ a_i \geq 0, \ n= \dim X.$$
  If $a_i >0$ for some $i$, then denoting ${\bf m}_x$ the maximal ideal at $x \in C$, we have
  $$H^0(C, N_{C \subset X}\otimes {\bf m}_x) \neq 0 \mbox{ while } H^1(C, N_{C \subset X} \otimes {\bf m}_x) = 0.$$
  By the basic deformation theory of submanifolds, this means that we can deform the rational curve $C$ fixing the point $x$.
  This is a contradiction to the assumption that  there are only finitely many smooth rational curves of degree $\ell$ through $x$. It follows that the normal bundle
  $N_{C \subset X}$ is trivial.  This implies that $\sW$ is an \'etale web of smooth rational curves.

  Now assume that the Picard group of $X$ is generated by an ample line bundle $L$. The anti-canonical bundle $K^{-1}_X$ of $X$ is isomorphic to $L^{{\rm i_X}}$ for an integer ${\rm i}_X$ (called the index of $X$) and the hyperplane line bundle of $\BP^N$ restricted to $X$ is isomorphic to $L^{ k}$ for a positive integer $k$.  Then for a general member $C$ of $\sW$,  $$C \cdot L^{k} = \ell \mbox{ and } C \cdot L^{{\rm i}_X} = 2.$$ It follows that ${\rm i}_X = \frac{2k}{\ell}$ and the anti-canonical bundle is ample.  For any \'etale web $\sV$ of smooth curves on $X$, a general member $C'$ of $\sW'$  has trivial normal bundle. Since $C' \cdot K^{-1}_X >0$, we see that $C'$ is a rational curve and $2 = C' \cdot L^{{\rm i}_X}.$  This implies that $C' \cdot L^{k} = \ell$.
  Thus $C'$ belongs to $\sW$.
  \end{proof}

We are ready for the proof of Theorem \ref{t.rational}.

\begin{proof}[Proof of Theorem \ref{t.rational}]
By Proposition \ref{p.example}, the manifold $X$ and $X'$ are Fano manifolds, thus they are simply connected.
From Proposition \ref{p.NBir} and  Proposition \ref{p.chain}, the webs $\sW$ and $\sW'$ satisfy the conditions (B) and (P) of  Theorem \ref{t.main}, from which Theorem \ref{t.rational} follows. \end{proof}

\section{Pleated webs}\label{s.pleat}

\begin{definition}\label{d.def}
Let $\sW$ be a web of curves on a projective variety $X$.  Write $P_a:= \mu_{\sW}(\rho_{\sW}^{-1}(a))$ for the curve in $X$ corresponding to  $a \in \sW$. Let $P_a \neq P_b, a, b \in \sW,$ be two distinct members through a point $x \in X_{\rm reg}$. Denote by $x(a)$ the unique intersection point $\mu_{\sW}^{-1}(x) \cap  \rho_{\sW}^{-1}(a)$ and by $P_b^{x(a)}$ the  unique irreducible component of $\mu_{\sW}^{-1}(P_b)$ through $x(a)$.  The image of the germ of the curve
$x(a) \in P_b^{x(a)}$ under $\rho_{\sW}$, which is a smooth germ of a 1-dimensional complex manifold through $a$ in $\sW$, will be  denoted by ${\rm Def}(P_a;P_b, x)$ and called the {\em deformation of $P_a$ along $P_b$ at $x$.} (The Zariski closure of ${\rm Def}(P_a; P_b, x)$ is the curve $\rho_{\sW}(P_b^{x(a)})$, which may have self-intersection at $a$. )
\end{definition}

\begin{lemma}\label{l.def}
Let $f:Y \to X$ be a generically finite morphism between projective varieties. Let $\sW$ be a web of curves and let
$f_{\flat}: f^*\sW \dasharrow \sW$ be the natural dominant rational map from the web $f^*\sW$ on $Y$ to $\sW$.
Let $Y_{\rm reg} \subset Y$ be the Zariski open subset with respect to $f^*\sW$ defined as in Proposition \ref{p.regular}.
Then there exists a dense Zariski open subset $Y_o \subset Y_{\rm reg} \cap f^{-1}(X_{\rm reg})$  such that
for any $a,b \in f^*\sW,  a \neq b,$ and $P_a \cap P_b \ni y \in Y_o,$
$$f_{\flat} {\rm Def}(P_a; P_b, y)  = {\rm Def}(f(P_a); f(P_b), f(y))$$ where the left hand side means the proper image under $f_{\flat}$.   \end{lemma}

\begin{proof}
From the definition of $f^*\sW$ in Definition \ref{d.pushpull} (ii), there exists a dense Zariski open subset ${\rm dom}(f_{\flat}) \subset f^*\sW$ such that
$f(P_a) = P_{f_{\flat}(a)}$ for any $a \in {\rm dom}(f_{\flat})$ and
we have the  commuting diagram
of morphisms $$ \begin{array}{ccccc} {\rm Univ}_{f^*\sW} & \supset & \rho_{f^*\sW}^{-1}({\rm dom}(f_{\flat}))  & \stackrel{{\rm univ}_{f^*\sW}}{\longrightarrow} & {\rm Univ}_{\sW} \\ \rho_{f^*\sW} \downarrow & & \downarrow & & \downarrow \rho_{\sW} \\ f^*\sW & \supset & {\rm dom}(f_{\flat})  &\stackrel{ f_{\flat}}{\longrightarrow} & \sW. \end{array} $$
 Choose $Y_o \subset Y_{\rm reg} \cap f^{-1}(X_{\rm reg})$ such that $\mu_{f^*\sW}^{-1}Y_o \subset \rho_{f^*\sW}^{-1}({\rm dom}(f_{\flat})).$
 If $P_a \cap P_b \ni y \in Y_o$ and $x = f(y)$, then $a, b \in {\rm dom}(f_{\flat})$ and   it is easy to see that  $$ {\rm univ}_{f^*\sW} (P_b^{y(a)}) = P_{f_{\flat}(b)}^{x(f_{\flat}(a))}.$$ This implies the desired result by the above commuting diagram. \end{proof}

\begin{definition}\label{d.pleat}
In the setting of Definition \ref{d.def}, let $x \neq y$ be two distinct points of $P_a \cap X_{\rm reg}.$ We say that $P_a$ is {\em pleated at}  $(x; y)$ if for any  $ b \in \sW$ with  $x \in P_b \neq P_a$, there exists $c \in \sW$ ($b=c$ allowed) satisfying  $$y \in P_c \neq P_a \mbox{  and } {\rm Def}(P_a; P_b, x) = {\rm Def}(P_a; P_c, y).$$ An irreducible component $\sV$ of $\sW$ is a {\em pleated component} of $\sW$ if for a general member $C$ of $\sV$ and a general point $x \in C$, there exists a point $y \in C \cap X_{\rm reg},  x \neq y,$ such that $C$ is pleated at $(x;y)$.
A web $\sW$ is {\em pleated} if it has a pleated component. A univalent web is pleated by definition. \end{definition}

\begin{proposition}\label{l.pleat}
Let $\sW$ be a pleated web on a projective variety $X$ with a pleated component $\sV$.
\begin{itemize} \item[(1)] Suppose that  for a member $C$ of $\sV,$ we have an infinite subset $\sZ \subset C \cap X_{\rm reg}$ and a point $y \in C\cap X_{\rm reg}$ such that $C$ is pleated at $(z;y)$ for  any $z  \in \sZ$.   Then $C$ is pleated at  $(z; z')$ for infinitely many pairs $z  \neq z'$ of  elements of $\sZ$.  \item[(2)]
Let $O \subset X$ be any dense Zariski open subset. Then  a general member $C$ of $\sV$ is  pleated at  $(x; y), x \neq y \in C\cap X_{\rm reg}$ such that $x, y \in O.$ \end{itemize}
\end{proposition}

\begin{proof} For each $x \in C \cap X_{\rm reg}$, denote by  $E^1_x, \ldots, E^{e}_x$  all the members of $\sW$ through $x$ different from $C$. For each $z \in \sZ$, as $C$ is pleated at $(z;y)$,  we have (not necessarily distinct) integers $1 \leq z(1), \ldots, z(e) \leq e$ such that
$$ {\rm Def}(C;E_z^{i},z) = {\rm Def}(C;E_y^{z(i)},y) \mbox{ for each } 1 \leq i \leq e.$$
Since there are only finitely many choices of $z(1), \ldots, z(e)$, we can assume that
$z(i) = z'(i)$ for each $1\leq i \leq e$ for infinitely many pairs $z \neq z' $ of elements of $\sZ$.
Then $${\rm Def}(C;E_z^{i},z) = {\rm Def}(C;E_y^{z(i)},y) =  {\rm Def}(C;E_y^{z'(i)},y) = {\rm Def}(C;E_{z'}^{i}, z')$$ for each $1 \leq i\leq e$. Thus $C$ is pleated at $(z;z')$.
 This proves (i).

Given $O \subset X$, choose a member $C$ of $\sV$ with $C \cap O \neq \emptyset$. Since $C$ is pleated, we have an infinite subset $\sZ  \subset C \cap O$ and a point $y_z \in C \cap X_{\rm reg}$ for each $z \in \sZ$ such  that $C$ is pleated at    $(z; y_z)$.   If infinitely many of $y_z$'s are distinct, then
 $y_z \in O$ for some $z \in \sZ$ and we are done. If infinitely many $y_z$'s coincide, we can deduce from (1) that $C$ is pleated at  $(z;z')$ for some pair $z, z' \in \sZ$,  proving (2).
\end{proof}

\begin{proposition}\label{e.pleat}
Let $f: Y \to X$ be a generically finite morphism between projective varieties. Let $\sW$ be a web of curves on $X$. Then   there exists a Zariski open subset $\sY \subset Y$ such that  a  member $C$ of ${\rm Mult}(f^*\sW)$ is pleated at   $(x; y)$ if  $x, y \in C \cap X_{\rm reg} \cap \sY$ satisfy  $x \neq y$ and $f(x) = f(y)$.  It follows that any component of  ${\rm Mult}(f^*\sW)$ is a pleated component of $f^*\sW$ and $f^*\sW$ is pleated if
${\rm Mult}(f^*\sW) \neq \emptyset.$ \end{proposition}

\begin{proof}  Let us use the terminology of Lemma \ref{l.def}. Using Lemma \ref{l.choose}, choose  a Zariski open subset $\sY \subset Y_o$ such that
if $P_a \cap \sY \neq \emptyset$ for $a \in \sW$, then $a \in {\rm dom}(f_{\flat})$ and $f_{\flat}$ is unramified at $a$.
Assuming that $x, y \in C \cap X_{\rm reg} \cap \sY$ satisfy  $x \neq y$ and $f(x) = f(y),$
let $E\neq C$ be any member of $f^*\sW$ satisfying $x \in E \cap C$ and let $F$ be an irreducible component (not necessarily different from $E$) of $f^{-1}(f(E))$ through $y$. By Lemma \ref{l.def},
$$f_{\flat} {\rm Def}(C;E,x) = {\rm Def}(f(C); f(E), f(x)) = $$ $$ {\rm Def}(f(C); f(F), f(y)) = f_{\flat} {\rm Def}(C;F,y).$$
As $f_{\flat}$ is unramified at the point corresponding to $C$ in ${\rm dom}(f_{\flat})$, this implies that
${\rm Def}(C;E,x) = {\rm Def}(C;F,y)$ as germs of curves in $\sW$.  It follows that $C$ is pleated at $(x;y)$. \end{proof}

\begin{proposition}\label{p.pleat}
Let $f: Y \to X$ be a generically finite morphism between projective varieties and let $Y_o \subset Y$ be as in Lemma \ref{l.def}.
Let $\sW$ be a web of curves on $X$ such that $f^*\sW$ is pleated with a pleated component $\sV$.
Assume that  for a general member $C$ of $\sV$ and a general point $z \in C \cap Y_o,$  there is a point $z' \in C \cap Y_o$ such that  $$f(z') \in X_{\rm reg}, \ z \neq z', \ f(z) \neq f(z') \mbox{ and } C \mbox{  is  pleated at } (z;z').$$ Then $f(C)$ is pleated at $(f(z);f(z'))$. It follows that $\sW$ is pleated, having  $f_*\sV$ as a pleated component. \end{proposition}

\begin{proof}
Any member of $\sW$ through $f(z)$, different from $f(C)$, is of the form $f(E)$ for some member $E \neq C$ of $f^*\sW$ through $z$. Since $C$ is pleated at $(z;z')$, there exists a member $F$ of $f^*\sW$ satisfying $F \neq C, z' \in F$ and
$${\rm Def}(C;E,z) = {\rm Def}(C; F, z').$$ By Lemma \ref{l.def},
 $${\rm Def}(f(C); f(E), f(z)) = f_{\flat} {\rm Def}(C; E, z) = $$ $$ f_{\flat} {\rm Def}(C; F ,z') = {\rm Def}(f(C); f(F), f(z')).$$
By the assumption $f(z) \neq f(z') \in X_{\rm reg}$, this implies that $f(C)$ is pleated at $(f(z); f(z'))$. \end{proof}

\begin{proposition}\label{p.pleat2}
Let $X$ be a simply connected projective manifold of Picard number 1 equipped with an \'etale web $\sW$ of smooth curves. Let $X'$ be a projective variety equipped with a web $\sW'$ of curves. Let $\Gamma \subset X \times X'$ be a generically finite algebraic correspondence  respecting $[\sW; \sW']$. If  $\Gamma$ is irreducible and the projection ${\rm pr}_X:  \Gamma \to X$ is not birational, then $\sW'$ is  pleated. \end{proposition}

\begin{proof}
To simplify the notation, write $Y= \Gamma$ and the two projections as $f: Y \to X$ and $q: Y \to X'$.
By assumption, the two webs $f^*\sW$ and $q^* \sW'$ on $Y$ coincide.
We know that ${\rm Mult}(f^* \sW)$ is nonempty from Proposition \ref{l.nbirational}. Let $\sV$ be any component of ${\rm Mult}(f^* \sW)$. Let $\sY \subset Y$ be as in Proposition \ref{e.pleat}.
For a general member $C$ of $\sV$ and a general point $z \in C \cap \sY $, we can choose a point $z'\in \sY$ such that
$$f(z) = f(z'), \ q(z) \neq q(z') \mbox{ and }  q(z') \in X'_{\rm reg}.$$ Then $C$ is pleated at $(z;z')$ by Proposition \ref{e.pleat}.
Thus $q_* \sV$ is a pleated component of $\sW'$ by Proposition \ref{p.pleat}. \end{proof}

\section{\'Etale webs of lines}\label{s.line}

\begin{definition}\label{d.line}
Let $X$ be a projective manifold. An \'etale web $\sW$ of smooth rational curves is called an {\em \'etale web of lines} if there is an embedding $X \subset \BP^N$ such that
the image of the members of the web are lines in $\BP^N$.
\end{definition}

\begin{proposition}\label{p.kollar}
Let $\sW$ be an \'etale web of lines on a projective manifold $X \subset \BP^N$.  A member $C \subset X$ of $\sW$ is said to be free if the normal bundle $N_{C \subset X}$ is trivial.
Let $\widehat{\sW}$ be the normalization of $\sW$ and let $\widehat{\sW}^{\rm free}$ be the dense Zariski open subset corresponding to free members.
Let $$\rho_{\widehat{\sW}}: {\rm Univ}_{\widehat{\sW}} \to \widehat{\sW} \mbox{ and }
\mu_{\widehat{\sW}}: {\rm Univ}_{\widehat{\sW}} \to X$$ be the normalization of $\rho_{\sW}$ and $\mu_{\sW}$.
Then \begin{itemize} \item[(i)] $\rho_{\widehat{\sW}}$ is a $\BP^1$-bundle;
\item[(ii)] $\widehat{\sW}^{\rm free}$ is contained in the smooth locus of $\widehat{\sW}$ and  the morphism $\mu_{\widehat{\sW}}$ is unramified on $\rho_{\widehat{\sW}}^{-1}(\widehat{\sW}^{\rm free}).$ \end{itemize}
    \end{proposition}

    \begin{proof}
    The normalization $\widehat{\sW}$  corresponds to the union of finitely many components of the normalized space of rational curves ${\rm RatCurves}^{\rm n}(X)$ defined in II.2.11 of \cite{Ko}. (i) is immediate because the members are lines (also  from II.2.12 of \cite{Ko}), while (ii) is from (i) and II.3.5.4 of \cite{Ko}. \end{proof}

   \begin{proposition}\label{p.normal}
   In the setting of Proposition \ref{p.kollar}, write $$P_a = \mu_{\widehat{\sW}}(\rho_{\widehat{\sW}}^{-1}(a)) \mbox{  for }
   a \in \widehat{\sW}.$$
  Fix an irreducible component $\sV$ of $\sW$ and let $f: {\rm Univ}_{\widehat{\sV}} \to  X$ and $g: {\rm Univ}_{\widehat{\sV}} \to \widehat{\sV}$ be the normalized universal family morphisms.
      Let $\sR \subset \widehat{\sV}$ be the  union of all irreducible hypersurfaces $H \subset \widehat{\sV}$ such that \begin{itemize} \item[(1)] $f(g^{-1}(H))$ is a hypersurface in $X$ and \item[(2)] the morphism $f$ is ramified  at a general point  of $g^{-1}(H).$ \end{itemize} Write $R = g^{-1}(\sR)$ and  $B= f(R).$  Then there exist a dense Zariski open subset $B_o \subset B$  and a dense Zariski open subset $O \subset X$  with the following properties.  \begin{itemize} \item[(i)] $B_o$ is
      contained in the smooth locus of $B$ and $g(f^{-1}(B_o) \cap R)$ is contained in the smooth locus of $\sR$.
    \item[(ii)]  If $a \in \sR$ and  $P_a  \cap B_o \neq \emptyset,$ then the normal bundle $N_{P_a \subset X}$ of $P_a \subset X$  is isomorphic to $\oplus_{1\leq i \leq n-1} \sO(m_i)$  for some integers $m_i$ satisfying $$m_1 \geq m_2 \geq \cdots \geq m_{n-2} \geq 0 > m_{n-1}.$$ \item[(iii)] In (ii), for any point $x \in P_a \cap B_o$,   the semipositive part of the fiber of the normal bundle of $P_a$ at $x$
   $$\oplus_{1 \leq i \leq n-2} \sO(m_i)_x \subset \oplus_{1 \leq i \leq n-1} \sO(m_i)_x \ \cong N_{P_a \subset X, x}$$  corresponds to  $$T_x(B_o)/T_x(P_a)
   \subset T_x(X)/T_x(P_a) = N_{P_a \subset X, x}.$$
   \item[(iv)] In (iii), if $s \in H^0(P_a, N_{P_a \subset X})$, then        $$s_x \in T_x(B_o)/T_x(P_a) \subset T_x(X)/T_x(P_a) =  N_{P_a \subset X, x}.$$ \item[(v)] Any member $E$ of $\sW$ with $E \cap O \neq \emptyset$ is a regular member of $\sW$, i.e., belonging to $\sW^{\rm etale}$,  satisfies $E \cap B = E \cap B_o$ and intersects $B_o$ transversally.
   \end{itemize}
   \end{proposition}

   \begin{proof}
      Denote by $\alpha : R \to B$ the restriction of $f$ and by $\beta: R \to \sR$ the restriction of $g$. Then  $\alpha$ is a generically finite morphism and we can choose a dense Zariski open subset $B_o \subset B$ contained in the smooth locus of $B$ such that $$\alpha|_{\alpha^{-1}(B_o)}: \alpha^{-1}(B_o) \to B_o$$  is \'etale and  the image $\beta(\alpha^{-1}(B_o))$ is contained in the smooth locus of $\sR$. For $a \in \beta(\alpha^{-1}(B_o))$,    write $N_{P_a \subset X} \cong \oplus_{i=1}^{n-1} \sO(m_i)$ with $m_i \geq m_{i+1}$. Since a general member of $\sW$ has trivial normal bundle, we know that $K_X^{-1} \cdot P_a = 2$ and
   $m_1 + \cdots + m_{n-1} =0$. The sections of the trivial normal bundle of $\beta^{-1}(a)$ inside $R$ give rise to elements of  $H^0(P_a, N_{P_a \subset X})$  generating the subspace $$T_x(B_o)/T_x(P_a) \subset  N_{P_a \subset X, x}$$ at each point  $x \in P_a \cap B_o.$
   It follows that $m_{n-2} \geq 0$. Then $m_{n-1} = - (m_1 + \cdots + m_{n-2})$ is negative unless $m_1 = \cdots = m_{n-1} =0$.
   But  the latter case cannot happen, because if  $N_{P_a \subset X}$ is a trivial bundle, then $f$ cannot be ramified at points of $g^{-1}(a)$ by Proposition \ref{p.kollar} (ii), a contradiction to the choice $a \in \sR$.  Thus $m_{n-1} <0$ and $\oplus_{1 \leq i \leq n-2} \sO(m_i)_x$ should correspond to $T_x(B_o)/T_x(P_a)$. This verifies (i)-(iv). (iv) is clear from Lemma \ref{l.positive}.
      \end{proof}

\begin{proposition}\label{p.line}
In the setting of Proposition \ref{p.normal}, assume furthermore that $X$ has Picard number 1.
Then $\sR, R, B$ and ${\rm Mult}(f^*\sW)$ are   nonempty. A general member of any component of ${\rm Mult}(f^*\sW)$ intersect $R$ transversally at  points of  $f^{-1}(B_o).$    \end{proposition}

\begin{proof}
Since $X$ has Picard number 1 and is covered by lines, it must be Fano, which implies that $X$ is simply connected. Thus the sets $\sR, R$ and $B$ in Proposition \ref{p.normal} are not empty and Proposition \ref{p.NBir} shows that ${\rm Mult}(f^* \sW)$ is not empty. For a general member $E$ of ${\rm Mult}(f^*\sW)$, the morphism $f|_E: E \to f(E)$ is a branched cover of a line.
Thus $E$ must intersect the ramification locus of $f$. Since $\sW$ is an \'etale web of lines, the images of ramification locus of $f|_E$ must cover a hypersurface in $X$ as $E$ varies in ${\rm Mult}(f^*\sW)$.  Thus $E$ has nonempty intersection with $R$. The intersection is transversal, because $f(E)$ should intersect $B$ transversally by  Lemma \ref{l.positive}. \end{proof}

\begin{theorem}\label{t.line}
Let $X\subset \BP^N$ be a projective submanifold of Picard number 1 and let $\sW$ be an \'etale web of lines on $X$.
Then $\sW$ is not pleated.   \end{theorem}

 \begin{proof}
 Let us assume that  $\sW$ has a pleated component $\sV$ and derive a contradiction. We will use the notation of Proposition \ref{p.normal} and Proposition \ref{p.line}.

  For a general member $E$ of ${\rm Mult}(f^*\sW),$ intersecting $R$ transversally at points of $f^{-1}(B_o)$, pick a point $z \in E$ with $f(z) \in O$ and let $C \subset X$ be the member of $\sV$ corresponding to $g(z)$. Since $C$ is pleated, there exists $y \in g^{-1}(g(z))$ such that $C$ is pleated at $(f(z); f(y))$. Moreover, we can assume that $f(y) \in O$ by Proposition \ref{l.pleat}. Thus we have a member $F \subset {\rm Univ}_{\widehat{\sV}}$ of $f^*\sW$ through $y$ such that $${\rm Def}(C; f(E), f(z)) = {\rm Def}(C; f(F), f(y)). $$ This implies that
  $g(E) = g(F)$ in $\widehat{\sV}$. This curve $g(E)$ intersects the hypersurface $\sR \subset \widehat{\sV}$ transversally at points in $g(f^{-1}(B_o) \cap R)$.

 Since $f(z)$ and $f(y)$ are two distinct points on the line $C$, the line $f(E)$ through $f(z)$ and the line $f(F)$ through $f(y)$ must be different.  Suppose that $f(E) \cap f(F) \neq \emptyset$.  Then the family of lines on $X$ parametrized by $g(E)$ lie on the plane
 $\langle f(E), f(F) \rangle \subset \BP^N$ spanned by $f(E)$ and $f(F)$. Thus  $\langle f(E), f(F) \rangle \subset X$. But then the line
 $C$ on this plane cannot have trivial normal bundle in $X$, a contradiction. Thus we have $f(E) \cap f(F) = \emptyset.$

 Let $\iota: \Delta = \{ t \in \C, |t| <1 \} \to g(E) \subset \widehat{\sV}$ be a local uniformization of the curve $g(E)$ at the point $\iota(0) \in g(E) \cap \sR$.
 The pull-back of the $\BP^1$-bundle $g:  {\rm Univ}_{\widehat{\sV}} \to \widehat{\sV}$ by $\iota$ is biholomorphic to  a trivial bundle
$p: \BP^1 \times \Delta \to \Delta$, equipped with a natural holomorphic map $j: \BP^1 \times \Delta \to {\rm Univ}_{\widehat{\sV}}$:
$$ \begin{array}{ccccc} \BP^1 \times \Delta & \stackrel{j}{\longrightarrow} & {\rm Univ}_{\widehat{\sV}} & \stackrel{f}{\longrightarrow} & X \\ p \downarrow & & \downarrow g & & \\ \Delta & \stackrel{\iota}{\longrightarrow} & \widehat{V}. & & \end{array} $$

 Write  $h: \BP^1 \times \Delta \to X$ for the composition $f \circ j$ such that for each $t \in \Delta$, the morphism $h_t: \BP^1 = \BP^1 \times \{ t \} \to X$ is an embedding as a line in $\BP^N$ and the line $h_0(\BP^1)  \subset X$ is contained in $B$.
Writing $u = h_0(\BP^1) \cap f(E)$ and $v = h_0(\BP^1) \cap f(F)$, we have $u \neq v$ from $f(E) \cap f(F) = \emptyset$.
Since the lines $f(E)$ and
$f(F)$ intersect $B_o$ transversally by the requirement $f(z), f(y) \in O$, we have $T_u(B_o) \cap T_u(f(E)) =0$ and $T_v(B_o) \cap T_v(f(F)) =0$.

Note that ${\rm d} h: T(\BP^1 \times \Delta) \to h^* T(X)$ sends the vertical tangent $T^p$ of the projection $p: \BP^1 \times \Delta \to \Delta$  into a line subbundle of  $h^*T(X),$ because $h_t$ is an embedding of $\BP^1$ to a line in $X$ for each $t$. The quotient bundle $\sN := h^*T(X)/T^p$ has the property that its restriction to the fiber $h^{-1}(t)$ is isomorphic to the normal bundle
$N_{h_t(\BP^1)\subset X}$.

The infinitesimal deformation $\frac{\partial}{\partial t} h$ defines a section $\sigma$ of the vector bundle $\sN$ on $\BP^1 \times \Delta$.
Let $k$ be the vanishing order of $\sigma$ along $t=0$, i.e. the nonnegative integer such that $t^{-k}\sigma$ is a holomorphic section of $\sN$ which does not vanish identically on $p^{-1}(0)$. Since the complex  analytic surface $h(\BP^1 \times \Delta)$ contains the germs of the lines $f(E)$ and
$f(F)$, the restriction
 $t^{-k}\sigma|_{h^{-1}(f(E))}$ (resp. $t^{-k}\sigma|_{h^{-1}(f(F))}$) takes values in $T(f(E))$ (resp. $T(f(F))$) modulo $T^{p}.$  On the other hand, Proposition \ref{p.normal} (iv) says that $t^{-k}\sigma|_{t=0}$ must take values in $T(B_o)$ modulo $T^{p}$. This
 implies that the values of $t^{-k}\sigma$ at $h_0^{-1}(u)$ and at $h_o^{-1}(v)$ must belong to $$T_u(B_o) \cap T_u(f(E)) = 0 \mbox{ and } T_v(B_o) \cap T_v(f(F)) =0.$$    Thus it gives a nonzero section of the normal bundle $$N_{h_0(\BP^1) \subset X} \subset N_{h_0(\BP^1) \subset \BP^N}$$ of the line
 $h_0(\BP^1) \subset X \subset \BP^N$, vanishing at the two distinct points $u$ and $v$. But the normal bundle of a line in $\BP^N$ cannot have a nonzero section vanishing at two
 distinct points, a contradiction. This proves Theorem \ref{t.line}. \end{proof}

\begin{proposition}\label{p.ell}
 In the setting of Theorem \ref{t.rational}, assume that   $\ell' =1$, then the projection $\Gamma \to X$ is birational, i.e., the generically finite correspondence $\Gamma$ defines a rational map $X \dasharrow X'.$ \end{proposition}

\begin{proof}
 Recall from Proposition \ref{p.example} that $\sW$ and $\sW'$ are \'etale webs of smooth curves.
If $\Gamma$ is not birational to $X$, then Proposition \ref{p.pleat} implies that the web $\sW'$ is pleated.
This is a contradiction to Theorem \ref{t.line}. \end{proof}

\begin{proposition}\label{p.easy}
In the setting of Theorem \ref{t.rational}, assume that $\Gamma$ is birational over both $X$ and $X'$, i.e., it defines a birational map $\Phi: X \dasharrow X'$. Then $\Phi$ gives a biregular morphism $X \cong X'$.
\end{proposition}

\begin{proof} The proof is essentially the same as that of Proposition 4.4 of \cite{HM01}, modulo a few minor changes.
We reproduce it for the reader's convenience.

Firstly, we claim that there is no hypersurface in $X$ (resp. $X'$)  contracted by $\Phi$ (resp. $\Phi^{-1}$). Let us prove it
for $\Phi$ (the same argument works for $\Phi^{-1}$). Assume the contrary and let $H\subset X$ be a hypersurface
contracted by $\Phi$, i.e., the proper image $\Phi(H)$ has codimension $\geq 2$ in $X'$. Since $X$ has Picard number 1, all general members of $\sW$ intersect $H$. Since $\Phi$ sends germs of members of $\sW$ to those of $\sW'$, the proper images under $\Phi$ of general members of $\sW$ give general members of $\sW'$. It follows that  all general members of $\sW'$ intersect the variety $f(H)$ of codimension $\geq 2$ in $X'$, a contradiction to the fact that $\sW'$ is an \'etale web of smooth curves.

By the claim, we see that $\Phi$ induces a biregular morphism between two quasi-projective varieties  $X_o \subset X$ and
$X'_o \subset X'$ such that the complement $X \setminus X_o$ and $X' \setminus X'_o$ have codimension $\geq 2$.
The isomorphism between the linear systems $H^0(X_o, K_{X_o}^{-k}) \cong H^0(X'_o, K_{X'_o}^{-k})$ induced by $\Phi$ for all $k >0$ extends to an isomorphism
$H^0(X, K_X^{-k}) \cong H^0(X', K_{X'}^{- k})$ by Hartogs extension. Since $X$ and $X'$ are Fano by Proposition \ref{p.example},  this isomorphism  gives a biregular morphism between $X$ and $X'$. \end{proof}

Now we are ready to prove Theorem \ref{t.ultim} and Theorem \ref{t.application}.

\begin{proof}[Proof of Theorem \ref{t.ultim}]
 Recall from Proposition \ref{p.example} that lines covering $X$ and $X'$ define  \'etale webs of lines.
Applying Theorem \ref{t.rational}, we have a generically finite correspondence $\Gamma \subset X \times X'$ extending $\varphi$. By Proposition \ref{p.ell}, we know that $\Gamma$ gives a rational map $X \dasharrow X'$. We can apply Proposition \ref{p.ell} with $X$ and $X'$ switched to  see that $\Gamma$ gives a birational map $\Phi: X \dasharrow X'$. Then Proposition \ref{p.easy} implies that $\Phi$ gives a biregular morphism. \end{proof}

\begin{proof}[Proof of Theorem \ref{t.application}] Let $\Phi: X \to X'$ be a surjective morphism.
 Recall from Proposition \ref{p.example} that lines covering $X$ (resp. smooth rational curves of degree $\ell$ covering $X'$) define an \'etale web $\sW$ (resp. $\sW'$) of smooth curves.   By Lemma \ref{l.coveretale}, the pull-back
   $\Phi^* \sW'$ is an \'etale web of smooth curves on $X$. The second assertion of Proposition \ref{p.example} implies that  $\Phi^*\sW'$ is a  subweb of $\sW$. Applying Proposition \ref{p.ell} to the webs $\sW'$ and $\Phi^* \sW'$, with $X$ and $X'$ switched,  we see that the graph  ${\rm Graph}(\Phi) \subset X \times X'$ must be birational to $X'$. Thus $\Phi$ is birational, which must be biregular by Proposition \ref{p.easy}.
\end{proof}

\bigskip
\noindent {\bf Acknowledgment} $\;$  I would like to thank the referee for pointing out an error in the proof of Theorem \ref{t.main} in the first version of the paper.

\end{document}